\documentclass[english]{amsart}	

\usepackage{amsmath}

\usepackage{amssymb}

\usepackage{amsthm}

\usepackage{epsfig}

\usepackage[T1]{fontenc}
\usepackage[latin9]{inputenc}
\usepackage{geometry}
\usepackage{verbatim}
\usepackage{float}
\usepackage{units}
\usepackage{textcomp}
\usepackage{mathrsfs}
\usepackage{amsthm}
\usepackage{amsmath}
\usepackage{amssymb}
\usepackage{graphicx}

\usepackage{pdflscape}

\usepackage{color}

\usepackage{babel}

\newtheorem{theorem}{Theorem}

\theoremstyle{plain}
\newtheorem{lemma}{Lemma}
\newtheorem{proposition}{Proposition}
\newtheorem{corollary}{Corollary}

\newtheorem{question}{Question}

\theoremstyle{definition}
\newtheorem{definition}{Definition}

\theoremstyle{remark}
\newtheorem{remark}{Remark}

\begin{document}

\title{Polynomial Time Relatively Computable Triangular Arrays in a Multinomial Setting}

\date{6 April 2016}

\author{Vladimir \mbox{Dobri\ensuremath{\acute{\text{c}}^{\dagger{ }}}} }
\address{Department of Mathematics, Lehigh University, 14 East Packer Avenue, 
Bethlehem, PA 18015}
\email{vd00@lehigh.edu}

\author{Patricia Garmirian}
\address{Department of Mathematics, Tufts University, 503 Boston Avenue, 
Medford MA 02155}
\email{Patricia.Garmirian@tufts.edu}

\author{Lee J. Stanley}
\address{Department of Mathematics, Lehigh University, 14 East Packer Avenue, 
Bethlehem, PA 18015}
\email[Corresponding author]{ljs4@lehigh.edu}

\thanks{Garmirian and Stanley dedicate this paper to the memory of
our dear departed mentor and friend, Vladimir \mbox{Dobri\ensuremath{\acute{\text{c}}}}}



\begin{abstract}
We extend the methods and results of {[}\ref{DSS}{]}
to the setting of
multinomial distributions satisfying certain properties.  These include
all the multinomial distributions arising from the direct proof of the 
Central Limit Theorem given in {[}\ref{DG}{]}, which, by results of that paper,
constitutes essentially full generality for the situations in which
the Central Limit Theorem holds.

\end{abstract}

\maketitle

\makeatletter

\section{Introduction}
\label{sec:1}
The two papers, {[}\ref{DG}{]} and {[}\ref{DSS}{]}, provide the context for our work here.  The first of these papers provides a direct proof of the Central Limit Theorem that proceeds by introducing a sequence of approximating multinomial distributions.  
The second deals more narrowly with the sequence $\left\{ S_n\right\}$ based on Rademacher random variables, so that
the basic setting is that of the binomial distributions with $p = \frac{1}{2}$.  It focuses on the notion of  {\it strong, trim } triangular array representations of the sequence of quantiles of the sequence $\left\{ S_n\right\}$ in this restricted setting, classifying them by sequences of {\it admissible permutations} and analyzing the complexity of the simplest of these classifying sequences.

We extend the methods and many (but not all) of the results of the latter to the setting of multinomial distributions provided by the former.  These multinomial distributions essentially capture (via sufficiently good approximation) the full generality of situations where the Central Limit Theorem holds:  this is one of the main results of {[}\ref{DG}{]}.  Thus, the results of this paper represent significant progress on the program laid out in the first paragraph of {[}\ref{DSS}{]}:  obtaining (and analyzing
the complexity of) {\it strong trim triangular representations for the sequences of quantiles} of weakly convergent sequences. 

As set forth in {[}\ref{DSS}{]}, a triangular array representation of a sequence $\left\{ X_n\right\}$ of random variables is {\it strong} if, for each
$n,\ X_n$ is equal pointwise, as a function, not just in distribution, to the sum of the random variables in the $n^{\text{th}}$ row
of the triangular array.  Thus (as was already noted in {[}\ref{DSS}{]}), there is very clear analogy:  strong triangular array
representations are to triangular array representations as almost sure convergence is to weak convergence.  Therefore, it is very
natural to seek strong triangular array representations for sequences that converge almost surely, as do the sequences of
quantiles of weakly convergent sequences.   It is reasonable to expect that a more detailed understanding
of the sequence of quantiles can be obtained via analysis in terms of a (particularly nice) strong triangular array representation thereof.

The issue of ``trimness'' will be discussed in more detail below, in the second paragraph of (1.1); for now, it will suffice to say that it addresses, for each $n$, the issue of how many bits in the binary expansion of $x \in (0,1)$
are required to determine the values of the random variables in the $n^{\text{th}}$ row of the triangular array.  
In the setting of {[}\ref{DSS}{]}, it was shown (Proposition 2 of that paper) that it is impossible
that each of the random variables in the $n^{\text{th}}$ row 
depends only on a single bit (even if the bit in question is allowed to vary with the random variable involved);
as noted there, this can also be seen via other routes.  The analogue of Proposition 2 of {[}\ref{DSS}{]}
holds in the setting of this paper, but we do not prove it here. 
In the setting of {[}\ref{DSS}{]}, the trimness notion was that each of the random variables in the $n^{\text{th}}$ row was to depend
only on the first $n$ bits; this was presented there as the natural ``next best hope'' as to how much information
about $x$ was needed.  The representations we obtain in this
paper will be ``trim'' in a sense that is the clear analogue of the trimness notion of {[}\ref{DSS}{]}:  the random variables in
the $n^{\text{th}}$ row will depend only on the first $n(M+1)$ bits, where $M$ is a parameter associated with the multinomial
distribution in question.  It was shown in {[}\ref{DSS}{]} (Proposition 1) that ``trimness does not come for free'':  there are
strong triangular array representations that are not trim; while this result goes over to the setting of this paper, via similar
arguments, we do not prove it here.  

As was noted in (1.3.1) of {[}\ref{DSS}{]}, there is another view of trimness, suggested
by A. Nerode (private communication); trimness is naturally associated with a Lipschitz continuity condition with $\delta = \epsilon$
for a transformation from $(0,1)$ to the the product of the $\{ -1, 1\}^n$ (with the product of the discrete topologies on
the $\{ -1, 1\}^n$).  This view of trimness carries over to the setting of this paper, by replacing each $\{ -1,1\}^n$ by
$\{ -1,1\}^{n(M+1)}$. 

The content of Corollary 1 of {[}\ref{DSS}{]} and the analogous result, Corollary 1 of this paper, is that strong trim triangular
array representations correspond exactly to sequences of admissible permutations.  Any admissible permutation can be
thought of as providing a ``rearrangement'' of $(0,1)$ (the domain of the random variables)  that
transforms the disorderly $S_n$ into the orderly step function, $S^*_n$.  In the setting of  {[}\ref{DSS}{]}, the rearrangement is literally a permutation of the dyadic intervals (of the depth appropriate for the permutation).  The situation is somewhat more complicated in this paper and is addressed in (2.1.1) and especially (2.1.2) where the sets $E_{n\ell}$ are introduced.  On each of these sets, $S_n$ is constant and each of these sets is the union of dyadic intervals of a fixed depth still greater than $n(M+1)$.     By permuting the second index ($\ell$), the rearrangement corresponding to a level $n$ admissible permutation maps the sets 
$E_{n,\ell}$ to depth $n(M+1)$ dyadic intervals,$D_{n,\ell '}$.

Further, the translation into the language
of admissible permutations makes it possible to pose and answer questions about the complexity of strong trim triangular
array representations in terms of the complexity of the corresponding sequence of admissible permutations.  These
measure the complexity of the rearrangements needed to transform the sequence $\left\{ S_n\right\}$ into the almost
surely convergent sequence of its quantiles.

We will say more, below, about the connections and differences between this paper, and the papers {[}\ref{DG}{]} and {[}\ref{DSS}{]}, but both of these end up working closely with a sequence, $\left ( R_n\vert n \in \mathbb{Z}^+\right )$, of i.i.d. random variables defined on $(0,1)$; both consider $\displaystyle{S_n := \sum_{i=1}^nR_n}$.  The main concern of {[}\ref{DSS}{]} was, for a very specific choice of $R_n$, to provide strong trim triangular array representations, $\left ( R^*_{n,i}\vert n \in \mathbb{Z}^+\text{ and } 1 \leq i \leq n\right )$  of the sequence, $\left ( S^*_n\vert n \in \mathbb{Z}^+\right )$ of {\it quantiles} of the sequence $\left ( S_n\vert  n \in \mathbb{Z}^+\right )$, and to analyze the complexity of the simplest of these representations.
We carry out this program here, starting from the more complicated and general sequences $\left ( R_n\vert n \in \mathbb{Z}^+\right )$ whose partial sums, $S_n$, are the multinomial distributions of {[}\ref{DG}{]}.

Echoing the above discussion, what made the triangular array representations of {[}\ref{DSS}{]} {\it strong} is that for each 
$\displaystyle{n,\ \sum_{i=1}^nR^*_{n,i}}$ is not only equal in distribution to $S^*_n$, but it is equal to it pointwise, as a function on $(0,1)$.  The representations $\left ( R^*_{n,i}\vert n \in \mathbb{Z}^+\text{ and } 1 \leq i \leq n\right )$ we obtain in
 this paper will be strong in exactly the same way.  The issue of ``trimness'' was also discussed above, and, as noted there,  will be discussed in more detail below, in the second paragraph of (1.1).

The first main result of {[}\ref{DSS}{]} was (for each $n$) to classify 
by {\it admissible }permutations of $\left \{ 0, \ldots , 2^n - 1\right\}$
all of the sequences $\left ( R^*_{n,i}\vert 1 \leq i \leq n\right )$
that sum to $S^*_n$ and satisfy the ``trimness'' condition along with some
additional properties.  The analogue of this result is Theorem 1 of this paper, proved in (2.2).  Now the classifying
permutations are permutations of $\left\{ 0, \ldots , 2^{n(M+1)} - 1\right\}$, where $M$ is a parameter associated
with one of the multinomial distributions arising from the setting of {[}\ref{DG}{]}, and admissibility is defined in an analogous
fashion appropriate to the multinomial distribution under consideration.  An immediate consequence is that
sequences, $\left ( \pi_n\vert n \in \mathbb{Z}^+\right )$, of admissible permutations classify strong trim
triangular array representations $\left ( R^*_{n,i}\vert n \in \mathbb{Z}^+\text{ and } 1 \leq i \leq n\right )$ of
$\left ( S^*_n\vert n \in \mathbb{Z}^+\right )$; this is Corollary 1, also (2.2).  

As was the case in {[}\ref{DSS}{]}, analysis of the admissibility
condition leads readily to the conclusion that there are ``many'' admissible permutations of $\left\{ 0, \ldots , 2^{n(M+1)} - 1\right\}$.  This is item (3) of Lemma 1, proved in (3.1).  Combining Theorem 1, Corollary 1 and Lemma 1 immediately gives Corollary 2, also in (3.1):  the existence of (continuum many) strong trim triangular array representations of $\left ( S^*_n\vert n \in \mathbb{Z}^+\right )$.
The second main result of {[}\ref{DSS}{]}  was the polynomial time computability of the sequence 
$\left ( F_n\vert n \in \mathbb{Z}^+\right )$,
of admissible permutations in which each $F_n$ is, in a suitable sense, the simplest possible 
permutation of $\left\{ 0, \ldots , 2^n - 1\right\}$; by extension, the
corresponding strong trim triangular array representation is also polynomial time computable.  Here too, we obtain
the analogous results:  these are Theorems 2, 3 in (3.2).   It would be too much to hope for polynomial time
computability outright, because of the dependence on the random variable $O$.  Instead we will have polynomial
time computability relative to a certain function $\tau^O_1$, introduced in Definition 4 in (3.1).  This function
encodes, in a very natural way, information about how level $n$ multinomial vectors map to the values of $S_n$.
We give a more detailed summary of the route to Theorems 2 and 3 at the beginning of \S 3.
\subsection{Preliminaries, Notation, Conventions}
\label{subsec:1.1}
Throughout this paper we work in the fixed probability space consisting
of the open real unit interval $(0,1)$ equipped with Lebesgue measure.
For $x \in (0,1)$, write:
\[
x = \sum_{i=1}^\infty \varepsilon_i(x)2^{-i}\text{, with each }\varepsilon_i(x) \in \{ 0,1\}\text{ and }
\varepsilon_i(x) = 0\text{ for infinitely many }i. 
\]

We work relative to 
a fixed random variable, $O$, with domain $(0,1)$.
In (1.2.1), below, we discuss the results of {[}\ref{DG}{]}, which
provide the context giving rise to 
$O$ and its properties.  For now, we lay out 
these properties as assumptions.

Associated with $O$ is a fixed non-negative integer, $M$.  We
assume that:
\[
\text{for }x \in (0,1),\ O(x)\text{ depends bijectively on }\left ( \varepsilon_1(x)\,\ \ldots \varepsilon_{M+1}(x)\right ).
\]
Thus, $O$ takes on $m = 2^{M+1}$ outcomes, which we enumerate in increasing order as   
$o_1,\ \ldots ,\ o_m$.  We make the further assumptions that these outcomes are equally likely,
that their sum is 0 and that the sum of their squares is 1.  
All of this will follow from the account, in (1.2.1), of the results of {[}\ref{DG}{]}; the displayed equation
for $O_M$ that occurs there is particularly relevant. 
Taken together, the increasing enumeration of the $o_s$ and the bijection between them and
the $\left ( \varepsilon_1(x)\,\ \ldots \varepsilon_{M+1}(x)\right )$ provide us with a fixed
enumeration, $\left ( \varsigma_s\vert 1 \leq s \leq m\right )$ of $\{ 0,1\}^{M+1}$:
\[
\varsigma_s\text{ is such that }O(x) = o_s\text{ for all }x\text{ such that } \left ( \varepsilon_1(x)\,\ \ldots \varepsilon_{M+1}(x)\right ) = \varsigma_s.
\]

The ``trimness'' condition of this paper is that a random variable, $Y$, defined on $(0,1)$ is {\it n-trim} if
and only if $Y$ depends only on $\left ( \varepsilon_1(y),\ \ldots ,\ \varepsilon_{n(M+1)}(y)\right )$.  Thus,
our assumption is that $O$ is {\it 1-trim}.

One of the main devices of {[}\ref{DG}{]} is to create an i.i.d family of ``copies'', $R_i$, of $O$, for positive
integers, $i$.  This is accomplished by taking $R_i(x)$ to be $O(P_i(x))$, where, for each $x \in (0,1),\
\left ( P_i(x)\vert i \in \mathbb{Z}^+\right )$ is a (uniformly defined) system of copies of $x$.
At a very high level, the $P_i$ can be defined by creating a pairwise disjoint
family $\left ( J_i\vert i \in \mathbb{Z}^+\right )$ of infinite subsets of $\mathbb{Z}^+$, 
and then ``taking $P_i(x)$ to be
the projection of $x$ onto the set of coordinates $J_i$.''  If $x$ were in
Cantor space rather than in $(0,1)$, this would be completely accurate; we
will make it precise in the next paragraph.  First, however, we note that while it is natural to think
that the $J_i$ should form a partition of $\mathbb{Z}^+$, in our implementation, this
will not be the case, and this will offer definite advantages.

For each positive integer, $i$, let $\left ( j_{i,k}\vert k \in \mathbb{Z}^+\right )$
be the increasing enumeration of $J_i$ and let:
\[
\overline{J}_i := \left\{ j_{i,1}, \dots , j_{i,M+1}\right\}\text{ and } \overline{J}^n := \overline{J}_1 \cup \ldots \cup
\overline{J}_n.
\]
We then take $P_i(x)$ to be that member of $(0,1)$, satisfying: 
\[
\text{for all }k,\ \varepsilon_k\left ( P_i(x)\right ) = \varepsilon_{j_{i,k} }(x).
\]

In {[}\ref{DG}{]}, $J_i$ is chosen to be ``the $i^{\text{th}}$ column of a triangular array of
positive integers'' (not to be confused with the triangular arrays of random variables
which are the main objects of study in this paper)
in which the ``rows'' {\it do }partition $\mathbb{Z}^+$,
where each row is a finite set of consecutive integers, and where
the sequence of row lengths, $\left ( \ell g_\rho\vert \rho \in \mathbb{Z}^+\right )$, 
is monotone non-decreasing and unbounded.   Note that with the stipulation that
the rows form a partition into finite sets of consecutive integers, the triangular array of
positive integers is completely specified by specifying the sequence of row lengths
$\left ( \ell g_\rho\vert \rho \in \mathbb{Z}^+\right )$.

\mbox{Dobri\ensuremath{\acute{\text{c}}}} and Garmirian adopt a specific choice of
$\left ( \ell g_\rho\vert \rho \in \mathbb{Z}^+\right )$ but note that, for the
results of {[}\ref{DG}{]}, any reasonable choice suffices.  For the purposes of this paper, and, in particular, for
the proof of Lemma 2, we make an apparently somewhat odd specific choice of 
$\left ( \ell g_\rho\vert \rho \in \mathbb{Z}^+\right )$.
For positive integers, $\rho$, we let:
\[
\ell g_\rho\text{ be the unique positive integer }b\text{ such that }(b-1)(M+1) < \rho \leq b(M+1).
\]
Thus, the rows come in ``blocks'' of
size $(M+1)$; for rows in the $b^{\text{th}}$ block, the row length is $b$.  We also
decompose a row index, $\rho$, as:
\[
(\ell g(\rho)-1)(M+1) + r(\rho)\text{, with }1 \leq r \leq M+1.
\]
Thus, $\ell g(\rho)$ is the block index of $\rho$ while $r(\rho)$ is the index of $\rho$ within
its block.

Via the $R_i$, for all positive integers, $n,\ O$ determines a multinomial distribution,
$\text{MN}^O_n$ with domain $(0,1)$.
The values, $\text{MN}^O(x)$, of the multinomial distributions are the  inner products, $k_1(x)o_1 + \ldots + k_m(x)o_m$,
where the $k_s(x)$ are the frequencies with which $o_s$ occurs in $\left ( R_1(x),\ldots , R_n(x)\right )$.  Thus, the
$k_s(x)$ are non-negative integers that sum to $n$ and so there are
exactly $\binom{n+m-1}{m-1}$  vectors $\left ( k_1(x), \ldots , k_m(x)\right )$ and at most this
many values of $\text{MN}^O_n$.

With all of this in place, we now``zoom in a bit more closely on'' our main objects of study.
For $x \in (0,1)$, and positive integers, $i, n$, we take:
\begin{equation}
R_i(x) := O\left ( P_i(x)\right )\text{ and }
S_n(x)  := \sum_{i = 1}^nO\left ( P_i(x)\right ) =  \sum_{i = 1}^nR_i(x).
\end{equation} 
As will be discussed in the next subsection, it follows from the results of {[}\ref{DG}{]}
that $S_n/\sqrt{n}$ converges (weakly) to the standard normal with domain
$(0,1)$.  Note that in view of our assumption on $O$, we have that:
\begin{equation}
\text{each }R_i\text{ depends only on the }j \in \overline{J}_i\text{ and so each }S_n\text{ depends only on
the }j \in \overline{J}^n.
\end{equation} 

As in {[}\ref{DSS}{]}, we are also interested in the quantiles, $S^*_n$, of $S_n$.  
The significance of the quantiles comes, in large measure,
from the work of Skorokhod.  He showed, {[}\ref{Skorokhod}{]}, that if 
$\left\{X_n\right\}$ is a sequence of random variables (on {\it any} probability
space) converging weakly to $X$, then the sequence of quantiles, $X^*_n$
(which are defined on $(0,1)$, regardless of the domain of the original $X_n$), 
converges almost surely to $X^*$, the quantile of $X$ (also defined on $(0,1)$).

Letting $C$ be the cumulative
distribution of $X_n$:
\begin{equation*}  X_n^*(x) := \text{inf} \left \{ t \in \mathbb{R}\vert C(t) \geq x\right \}.
\end{equation*}
It is well-known that $X_n$ and $X_n^*$ are equal in distribution.
It then follows from the results of {[}\ref{DG}{]}, that:
\begin{equation*}  \text{The sequence }\{ S^*_n/\sqrt{n}\}\text{ converges almost surely to
the standard normal with domain }(0,1).
\end{equation*}

We will
focus on representations:
\[
S^*_n = \sum^n_{i=1}R^*_{n,i}\text{ where each }R^*_{n,i}\text{ is }n-\text{trim}
\]
(and satisfies a number of other properties laid out in the statement of Theorem 1).
The admissibility condition for permutations of $\left\{ 0, \ldots , 2^{n(M+1)}-1\right\}$ will be given in (2.1.4), below.

For $n > 0$ and non-negative integers $\ell < 2^{n(M+1)}$, we let $D_{n,\ell}$ be the $\ell^{\text{th}}$ depth 
$n(M+1)$ dyadic interval, i.e.:
\[
D_{n,\ell} = 
\begin{cases}
     \left ( 0,2^{-n(M+1)}\right )\text{ if }\ell = 0,\\
    \left [ \ell 2^{-n(M+1)}, (\ell + 1)2^{-n(M+1)}\right )\text{ otherwise.}
\end{cases}
\]
We argue, in (2.1.1.) below, that each $S^*_n$ is constant on each $D_{n,\ell}$.  As in {[}\ref{DSS}{]}, this
will allow us to define ``integer versions'', $\text{I}S^*_n$, of the $S^*_n$:
\[
\text{I}S^*_n(\ell ) =\text{ the constant value of }S^*_n\text{ on } D_{n,\ell}\text{, for }0 \leq \ell < 2^{n(M+1)}.
\]

As in {[}\ref{DSS}{]}, this serves as a paradigm case for introducing integer versions of notions (usually functions, but sometimes
relations) naturally defined on $(0,1)$, when the notion is invariant on a system of subsets of $(0,1)$, indexed by integers.
The indexing will typically also depend on $n$ as a parameter, and the subsets in question will typically be either the $D_{n,\ell}$ or the $E_{n,\ell}$  (the analogues, for $S_n$, of the $D_{n,\ell}$ 
which we introduce in (2.1.2), below). 

Our approach to complexity issues will follow that of {[}\ref{DSS}{]}, especially as laid out
in the final two paragraphs of subsection (1.2) of that paper.  We also include
the references {[}\ref{Clote}{]}, {[}\ref{Papa}{]}
from that paper for general background on
computational complexity.

\subsection{Connections with {[}\ref{DG}{]} and {[}\ref{DSS}{]}}
\label{subsec:1.2}
In this subsection we elaborate further on the connections between this paper and the work of {[}\ref{DG}{]} (in (1.2.1))
and that of {[}\ref{DSS}{]} (in (1.2.2)). 
\subsubsection{The Context Provided by the Results of \textnormal{{[}\ref{DG}{]}}}
\label{subsubsec:1.2.1}
In {[}\ref{DG}{]}, Dobri\'c and Garmirian give an explicit  proof of the CLT directly from the definition of 
weak convergence, which states that (with $S$ a Polish space equipped
with the $\sigma-\text{algebra }\mathcal{B}(S)$ of Borel subsets of $S$) a sequence of measures $\mu_n$ (on $S$)
converges to $\mu$ weakly provided that for each bounded, continuous 
function $f:S\rightarrow{\mathbb{R}}$,
\[
\lim_{n\rightarrow{\infty}}\int\, f(x)\, d\mu_n(x)=\int\, f(x)\, d\mu(x).
\]
In the setting of {[}\ref{DG}{]}, $S$ is $\mathbb{R}$  and $\left\{\mu_n\right\}$ is the sequence of measures 
induced by a specified i.i.d. sequence of random variables.

The starting point in {[}\ref{DG}{]} is not the random variable $O$ of (1.1), but a much more general random variable, $Q$, on
which the only assumptions are that $Q$ has mean 0 and variance 1.  The random variable $O$ of (1.1) will be obtained
as a ``truncated version'' of the expansion of $Q$ with respect to the Haar basis.  Such an expansion is available
since $Q\in L^2(0,1)$.

Using that paper's verions of the sort of $P_i$ described in (1.1), an i.i.d. sequence $\left ( Q_i\vert i \in \mathbb{Z}^+\right )$ of copies
of $Q$ is obtained by taking $Q_i(x)$ to be $Q\left ( P_i(x)\right )$.  Expanding $Q$, respectively the $Q_i$,
with respect to the Haar basis yields:
\[
Q(x)=\sum_{k=0}^\infty 2^{\frac{k}{2}}c_{k,\lfloor2^k x\rfloor}(-1)^{\varepsilon_{k+1}(x)},
\]
\[
Q_i(x)=\sum_{k=0}^\infty 2^{\frac{k}{2}}c_{k,\lfloor2^kP_i(x)\rfloor}(-1)^{\varepsilon_{k+1}(P_i(x))}.
\]

The next step in {[}\ref{DG}{]} is to consider the following sums of $n$ i.i.d. random variables:
\[
\Sigma_n(x) := \sum_{i=1}^nQ_i(x).
\]
In {[}\ref{DG}{]}, these $\Sigma_n$ were denoted by $S_n$.  The following truncated versions 
of $Q$, the $Q_i$ and $\Sigma_n$ were also
considered in {[}\ref{DG}{]}; the last was denoted there by $S_{n,M}$.  For 
$M\in{\mathbb{N}}$, define:
\[
O_M(x) := \sum_{k=0}^M2^{\frac{k}{2}}c_{k,\lfloor2^kx\rfloor}(-1)^{\varepsilon_{k+1}(x)}.
\]
It follows from properties of the outcomes that $O_M$ will also have mean 0 and that
\[
 Var\left ( O_M\right ) = \theta^2_M = \sum_{j=0}^M\sum_{k=0}^{2^j-1}c_{j,k}.
\]
In {[}\ref{DG}{]}, $\sigma^2_M$ is used in place of $\theta^2_M$.

For positive integers, $i$, define:
\[
R_{i,M}(x) := O_M\left ( P_i(x)\right ) \left ( =  \sum_{k=0}^M2^{\frac{k}{2}}c_{k,\lfloor2^kP_i(x)\rfloor}(-1)^{\varepsilon_{k+1}(P_i(x))}\right )\text{ and: }
\]
\[
\Sigma_{n,M}(x):=\sum_{i=1}^nR_{i,M}(x).
\]

Dobri\'c and Garmirian showed that for each bounded and continuous $f:\mathbb{R} \to \mathbb{R}$ and for each $\epsilon > 0$, there exists a natural number $M_0$ such that for all $M \geq M_0$, with $\theta_M$ as above, 
\[
\left|\int_0^1f\left ( \frac{\Sigma_n(x)}{\sqrt{n}}\right )\, d\lambda (x) - \int_0^1f\left(\frac{\Sigma_{n,M}(x)}{\theta_M\sqrt{n}}\right )\, d\lambda (x)\right| < \epsilon.
\]
Thus, they showed that these Haar expansions are in fact ``sufficiently close'' to their truncated versions.  They then further approximated these truncated versions by the standard Gaussian measure on $\mathbb{R}$, completing their proof of the CLT.

It is then clear that for each $i$, each $M$ and for all $x$:
\begin{equation}
R_{i,M}(x)\text{ depends only on }\left (\varepsilon_{j_{i,k}}(x)\vert 1 \leq k \leq M+1 \right )\text{ and so} 
\end{equation}
\begin{equation}
\text{for all }x,\ \Sigma_{n,M}(x)\text{ depends only on }\left (\varepsilon_{j_{i,k}}(x)\vert 
1 \leq k \leq M+1, 1 \leq i \leq n \right ).
\end{equation}

The context for this paper is then provided by fixing a suitable $M$, letting $O$ be $O_M$, 
 and (referring back to Equation (1)) letting
$R_i = O\circ P_i = R_{i,M}$ and $S_n$ be $\Sigma_{n,M}$, but with the specific choice of the $P_i$ described
in (1.1).

\subsubsection{The Special Case of \textnormal{{[}\ref{DSS}{]}}:  $M = 0$}
\label{subsubsec:1.2.2}  Here, we indicate how the work of {[}\ref{DSS}{]} can and should be understood as the
analysis of the special case $M = 0$, and also discuss similarities and some differences in the approach taken in
this paper.  The {[}\ref{DSS}{]} analogue of the random variable $O$ should be one of two Rademacher random variables  
defined on $(0,1)$, taking on values -1 and 1, each with probability 1/2, and depending only on $\varepsilon_1$.
This brings us to a minor discordance between the treatment in {[}\ref{DSS}{]} and that of {[}\ref{DG}{]}:  in {[}\ref{DSS}{]}, implicitly, $O(x)$
is taken to be $\displaystyle{(-1)^{1+\varepsilon_1(x)}}$, so that the natural order on $(0,1)$ coincides with
the lexicographic order on $\{ -1, 1\}^{\mathbb{Z}^+}$,  whereas the true analogue of the treatment in {[}\ref{DG}{]} would take
$O(x)$ to be $\displaystyle{(-1)^{\varepsilon_1(x)}}$.  Here, we have opted for the latter choice, so that the analogy
with {[}\ref{DSS}{]} is imperfect in this minor respect.  The role of $O$ in {[}\ref{DSS}{]} is left in the background, however, mainly because 
the notions developed above depend so much more transparently on $x$ in the setting of that paper.  This will be a recurring theme, and we shall comment on its influence in various places in what follows.

Continuing with the system of analogies, the {[}\ref{DSS}{]} analogue of the triangular array of integers is simply
the diagonal:  each ``row'' has length 1 and $i$ is the unique entry in the $i^{\text{th}}$ row.  Thus,
(in view of the observations in the previous paragraph) the {[}\ref{DSS}{]} analogue of $R_i$ is  $\displaystyle{R_i(x) = 
(-1)^{1+\varepsilon_i(x)}}$.  The analogue of $P_i$ is denoted by $X_i$ in {[}\ref{DSS}{]}, and $X_i(x)$ is simply
$\varepsilon_i(x)$.  In {[}\ref{DSS}{]}, the $R_i$ are simply taken to be given, as in the previous sentence, 
rather than ``created'' as copies of $O$
via the various $P_i$.  

In view of the preceding, note that the {[}\ref{DSS}{]} analogue of $\overline{J}_i$ is simply $\{ i\}$.  One motivation for the
choice we have made, above, of the triangular array of integers, is to try to recapture as much as possible
of the simplicity of this situation; our $\overline{J}_i$ also are not only pairwise disjoint, but come in
consecutive blocks:  if $i_1 < i_2$ then all members of $\overline{J}_{i_1}$ are less than
all members of $\overline{J}_{i_2}$.  In fact, specifying this property along with the specification
that the cardinality of each $\overline{J}_i$ is $M+1$ almost completely determines the choice of the $P_i$,
we have made in this paper and described in (1.1),
in that it requires that each row length is repeated {\it at least, but not necessarily exactly }$M+1$ times.  After the
proof of Lemma 2, in (3.2), we will indicate exactly how this
faciilitates the proof of this Lemma.
\section{Theorem 1 and Corollary 1}
\label{sec:2}
\subsection{Preliminaries for Theorem 1}
\label{subsec:2.1}
\subsubsection{Level Sets for $S^*_n$}
\label{subsubsec:2.1.1}
As in {[}\ref{DSS}{]}, the level sets for $S^*_n$ are unions of dyadic intervals, but here they are of depth $n(M+1)$, the
$D_{n,\ell}$ introduced in the penultimate paragraph of (1.1).
This is because the intervals on which $S^*_n$ is constant have lengths which are integer multiples of $2^{-n(M+1)}$.
This, in turn, follows from the fact that $S^*_n$ is the inverse of the cdf of $S_n$.
\subsubsection{Level Sets for $S_n$}
\label{subsubsec:2.1.2}
Here, as a result of the multinomial distribution and the related ``shuffle'' involved in the definition of the $P_i(x)$,
the situation is rather different.  The level sets for $S_n$ will not be unions of depth $n(M+1)$ dyadic
intervals; greater depth is necessarily involved once we leave behind the simple $M = 0$ case of {[}\ref{DSS}{]}.  The correct picture emerges from a more detailed analysis which we now undertake.

Recall the $J_i$ from (1.1) and that their increasing enumerations
are the $\left ( j_{i,k}\vert k \in \mathbb{Z}^+\right )$. Recall also from there the $\overline{J}_i$ and
the $\overline{J}^n$.
Also, as noted in Equations (2) and (3), above, 
the sequence $\left ( \varepsilon_{j_{i,k}}(x)\vert 1 \leq k \leq M+1\right )$ completely determines the value
of $R_i(x)$.  

Now, fix $n \geq 1$.  The sets $\overline{J}_i$,  for $1 \leq i \leq n$, are (of course) pairwise
disjoint and therefore, their union, $\overline{J}^n$, has cardinality $n(M+1)$.   So, as noted in Equations (3) and (4), above,
taken together, $\left ( \varepsilon_{j_{i,k}}(x)\vert
1 \leq i \leq n,\ 1 \leq k \leq M\right )$ completely determines the sequence $\left ( R_i(x)\vert 1 \leq i \leq n\right )$,
and therefore determines $S_n(x)$, the sum of this sequence.  Let 
$\left ( \vec{e}_\ell\vert 0 \leq \ell < 2^{n(M+1)}\right )$
be the increasing enumeration of $\{ 0,1\}^{\overline{J}^n}$ with respect to lexicographic order; thus, each
$\vec{e}_\ell$ is a bitstring with domain $\overline{J}^n$; we'll denote its value at a coordinate $j \in \overline{J}^n$ by
$\left ( \vec{e}_\ell\right )_j$.
Finally, define
\[
E_{n,\ell} := \left \{ x \in (0,1)\vert \varepsilon_j(x) =  \left ( \vec{e}_\ell\right )_j\text{ for all } j \in \overline{J}^n\right \}.
\]
Then, each $E_{n,\ell}$ has measure $2^{-n(M+1)}$ and the level sets of $S_n$ are unions of these $E_{n,\ell}$.

\subsubsection{Sequences of binary digits and outcomes}
\label{subsubsec:2.1.3}
For $x \in (0,1)$, we set: 
\[
\vec{b}_n(x) := 
\left ( \varepsilon_1(x), \ldots , \varepsilon_{n(M+1)}(x)\right ),
\]
\[
\vec{o}_n(x) := 
\left ( R_1(x), \ldots , R_n(x)\right ).
\]
For $1 \leq \ell \leq n,\ \vec{b}_n(x)$ is constant on $D_{n,\ell}$ and $\vec{o}_n(x)$  is constant on $E_{n,\ell}$.  Therefore, following
the convention in the final paragraph of (1.1), we denote 
these constant values by $\text{I}\vec{b}_n(\ell ),  \text{I}\vec{o}_n(\ell )$,
respectively.  Of course, the same also applies to each individual $R_i$ and so we
also denote by $\left( \text{I}R_i\right )_n(\ell )$ the constant value of $R_i(x)$ on $E_{n,\ell}$.
As in {[}\ref{DSS}{]},
$\text{I}\vec{b}_n(\ell )$ is the reverse of the binary representation of 
$\ell$:
\[
\ell = \sum_{i=1}^{n(M+1)} 2^{n(M+1)-i}\left ( \text{I}\vec{b}_n(\ell )\right )_i,
\] 
and $\left ( \text{I}\vec{b}_n(\ell )\vert\, \ell < 2^{n(M+1)}\right )$
enumerates $\{ 0, 1\}^{n(M+1)}$ in increasing order with respect to the lexicographic 
ordering.  For $\vec{b} \in \{ 0, 1\}^{n(M+1)}$, we also let:
\[ 
D_{\vec{b}} := \{ x \in (0,1)|\vec{b}_n(x) = \vec{b}\}.  
\]
Letting $\ell$
be such that $\vec{b} = \text{I}\vec{b}_n(\ell )$, we note that $D_{\vec{b}} = D_{n,\ell} = 
\left\{ x \in (0,1)\vert \vec{b}_n(x) = \text{I}\vec{b}_n(\ell )\right\}$.  

Note that $\text{I}\vec{b}_n(\ell )$ does NOT denote the $\ell^{\text{th}}$ component of a vector,
$\text{I}\vec{b}_n$, rather it denotes the vector (a length
$n(M+1)$ bitstring) itself.  We will denote the $i^{\text{th}}$ component
of this vector by $\left ( \text{I}\vec{b}_n(\ell )\right )_i$.  Note that
this is just $\varepsilon_i(x)$ for any $x \in D_{n,\ell}$.  Similar observations
hold with $\vec{o}$ in place of $\vec{b}$ (and the set of outcomes of $O$  replacing $\{ 0, 1\}$).

Let $\mathbf{O_n}$ denote the set of vectors $\left\{\text{I}\vec{o}_n(\ell )\vert 0 \leq \ell < 2^{n(M+1)}\right \}$ and
note that this gives an enumeration without repetitions of $\mathbf{O_n}$  (this is essentially because there is a bijection
between outcomes of $O$ and subsets of $\{ 1, \ldots , M+1\}$).  Thus $\vec{b}_n(\ell ) \mapsto \vec{o}_n(\ell )$
is a bijection between $\{ 0,1\}^{n(M+1)}$ and $\mathbf{O_n}$.  We use $\nu_n$ to denote this bijection. 
 Note that for $\vec{b} \in \{ 0, 1\}^{n(M+1)},\ x \in D_{\vec{b}}$ and $1 \leq i \leq n,\ 
\left ( \nu_n(\vec{b})\right )_i = R_i(x)$.

\subsubsection{Admissible Permutations}
\label{subsubsec:2.1.4}
We are now ready to define this paper's version of the notion of admissible permutation 
(of $\left \{ 0, \ldots, 2^{n(M+1)} - 1\right\}$) and to prove Theorem 1, which is best understood
in the following way.  As already noted, each $E_{n,\ell}$ tells us much more than a value of $S_n$:  it tells
us {\it how this value arises as the sum of values of the} $R_i$ for $1 \leq i \leq n$.  The content of Theorem 1 amounts to unravelling the analogous additional information about how a value of $S^*_n$ arises:  as the sum of values of the members of an i.i.d. family of $n$ random variables, each of which is $n-\text{trim}$, i.e., as the sum of the $ R^*_{n,i}$ for $1 \leq i \leq n$.

\begin{definition}  A permutation, $\pi$, of $\left \{ 0, \ldots, 2^{n(M+1)} - 1\right\}$ is {\it admissible}
iff for all $0 \leq \ell < 2^{n(M+1)}$:

\[
S^*_n(x) = S_n(y)\text{ for all }x \in D_{n,\ell }\text{ and all }y \in E_{n,\pi (\ell )}.
\]
\end{definition}

\subsection{Theorem 1 and Corollary 1}
\label{subsec:2.2}
\begin{theorem}  For each $n$, there is a canonical bijection
between admissible permutations of 
$\left\{ 0, \ldots , 2^{n(M+1)}-1\right \}$ and representations of $S^*_n$ as a sum
$\displaystyle{S^*_n = \sum_{i=1}^nR^*_{n,i}}$,
where $\left ( R^*_{n,i}\vert\, 1 \leq i \leq n\right )$ is an
i.i.d. family of $n-\text{\textnormal{trim}}$ random variables each of which has mean 0, variance 1 and has outcomes of $O$
as values, with equal probability.
\end{theorem}
\begin{proof}
Fix $n > 0$.  We first construct the $R^*_{n,i}$, given $\pi$, and then construct $\pi$ given the  
$R^*_{n,i}$.  We then carry out the necessary verifications in each direction. 

Let $\pi$ be an admissible permutation of
$\left\{ 0, \ldots , 2^{n(M+1)}-1\right\}$.  For
$x \in (0, 1)$, let $\ell$ be such that $x \in D_{n,\ell}$, and let $1 \leq i \leq n$.    
Then:    
\begin{equation}
\text{let } y \text{ be any member of } E_{n,\pi (\ell )} \text{ and define:  }  R^*_{n,i}(x) := R_i(y).
\end{equation}
\noindent

Conversely, given an independent family, $\left ( R^*_{n,i}\vert 1 \leq i \leq n\right )$, such that
$\displaystyle{S^*_n = \sum_{i=1}^nR^*_{n,i}}$, where
each $R^*_{n,i}$  is $n-\text{trim}$, has mean 0, variance 1 and has outcomes of $O$
as values, with equal probability, we obtain $\pi$ as follows.  Given $0 \leq \ell < 2^{n(M+1)}$, let
$x \in D_{n,\ell}$, let $\vec{o} = \left ( R^*_{n,i}(x)\vert 1 \leq i \leq n\right )$ and
define:
\begin{equation}
\pi (\ell )\ =\ \text{that } \ell '\text{ with } 0 \leq \ell ' < 2^{n(M+1)}\ \text{such that } \vec{o} = \text{I}\vec{o}_n\left ( \ell '\right ).
\end{equation}
Clearly these constructions yield a bijection, so we turn to the necessary verifications.

First suppose $\pi$ is admissible and that the 
$R^*_{n,i}$ are defined by Equation (5). 
Clearly these $R^*_{n,i}$ are $n-\text{trim}$, have mean 0 and variance 1.
In order to see that they sum to $S^*_n$, note that:
\[
\text{for all } \ell < 2^{n(M+1)} \text{ and all } x \in D_{n,\ell},\
S^*_n(x) = S_n(y) \text{,  for any } y \in E_{n,\pi (\ell )},
\]
\[
\text{i.e. } S^*_n(x)
= \sum_{i=1}^nR_i(y) \text{, for any such } y,
\text{i.e. } S^*_n(x) = \sum_{i=1}^nR^*_{n,i}(x);\text{ this suffices.}
\]

We next show that for each $1 \leq i \leq n$ and each of the $m$ outcomes, $o$ of $O,\ 
P\left ( R^*_{n,i} = o\right ) = \frac{1}{m} = 2^{(n-1)(M+1)}\times 2^{-n(M+1)}$ by showing
that the event ``$R^*_{n,i} = o$'' is the union of $2^{(n-1)(M+1)}$ dyadic inteverals $D_{n,\ell}$.
For this, note that the event ``$R_i = o$'' \emph{does} have probability $\frac{1}{m}$
and therefore is the union of $2^{(n-1)(M+1)}$ of the sets $E_{n,\ell}$.  It is routine to see
that a set $E_{n,\ell}$ is included in the event ``$R_i = o$'' iff the set 
$\displaystyle{D_{n,\pi^{-1}(\ell )}}$ is included in the event
``$R^*_{n,i} = o$'' and so this suffices.

In order to see that these $R^*_{n,i}$ are independent,
it suffices to show that:
\[
\text{for all } \vec{o} = \left ( o_1, \ldots , o_n\right ) \in \mathbf{O}_n,\
p\left(o_1,\ldots,o_n\right) = 
p_1\left(o_1\right)\cdot\ldots\cdot p_n\left(o_n\right),
\]
where $p$ is the joint pmf of the $R^*_{n,i}$ and each $p_i$
is the pmf of $R^*_{n,i}$ alone.  
We have already seen that $p_1\left(o_1\right)\cdot\ldots\cdot p_n\left(o_n\right) =
2^{-n(M+1)}$, so,
again viewing $\pi$ as a permutation
of $\{ 0, 1\}^{n(M+1)}$, let 
$\vec{b} := \pi^{-1}\circ \left ( \nu_n\right )^{-1}(\vec{o})$ and note that:
\[
P\left(R^*_{n,1}=o_1,\ldots,R^*_{n,n}=o_n\right) = 
\lambda\left(\left\{ x\vert \pi\left ( \vec{b}_n(x)\right ) = 
\nu_n^{-1}(\vec{o}\right\} \right) = 
\lambda\left(D_{\vec{b}}\right) = 2^{-n(M+1)}.  
\]

For the opposite direction, suppose that $\left ( R^*_{n,i}\vert 1 \leq i \leq n\right )$
is given with the stated properties.  Let $\pi$ be defined
by Equation (6). 
We first show that $\pi$ is one-to-one.  For this, let $x\in D_{n,\ell},\ 
\vec{o} = \left ( o_1, \ldots , o_n\right ) =  \left ( R^*_{n,i}(x)\vert 1 \leq i \leq n\right )$ and note that if
$\vec{u} \in \{ 0, 1\}^{n(M+1)}$ is such that 
\[
\text{for } y \in D_{\vec{u}},\ \left ( R^*_{n,i}(y)\vert 1 \leq i \leq n\right ) =
\vec{o} \text{, then } \vec{u} = \text{I}\vec{b}_n(\ell ).  
\]
If this
were to fail we would have that 
\[
P\left(R^*_{n,1}=o_1,\ldots,R^*_{n,n}=o_n\right) \geq\  
\lambda\left(D_{\vec{u}}\right)+
\lambda\left(D_{n,\ell }\right)=2^{-n(M+1)+1}, 
\]
which contradicts
our hypotheses on the $R^*_{n,i}$.  Thus, $\pi$ is one-to-one.  Admissibility
then follows, because now, by hypothesis, if $x \in D_{n,k}$ and $m = \pi (k)$, then:
\[
S^*_n(x) = \sum_{i=1}^nR^*_{n,i}(x) = \sum_{i=1}^no_i \text{, but also 
for any } y \in D_{n,m},\ S_n(y) = \sum_{i=1}^no_i,
\]
\noindent
as required.  
\end{proof}
\begin{corollary}  There is a canonical bijection between
sequences, $\left\{ \pi_n\right\}$,
of admissible permutations of $\left\{ 0, \ldots , 2^{n(M+1)}-1\right \}$ and trim,
strong triangular arrays for 
$\left\{ S^*_n\right\}$.
\qed
\end{corollary}
Given $n$ and $\left ( R^*_{n,i}\vert 1 \leq i \leq n\right )$, Equation (6) 
is best seen as as a finer version of the displayed formula of Definition 1 (the definition of admissible permutation).
Incorporating the additional information in the representations $\left ( R_i\vert 1 \leq i \leq n\right )$
and $\left ( R^*_{n,i}\vert 1 \leq i \leq n\right )$ singles out
a specific admissible permutation, whereas the displayed formula of  Definition 1 defines the set  of all of them.
Equation (5) reverses this, taking as given the canonical representation, 
$\left ( R_i\vert 1 \leq i \leq n\right )$, together with a specific admissible permuation 
and singling out a specific representation,
$\left ( R^*_{n,i}\vert 1 \leq i \leq n\right )$, of $S^*_n$.
\section{Theorems 2 and 3}
We follow the general outline of \S 3 of {[}\ref{DSS}{]}, omitting
certain items, and in particular, those specifically targetted
at \S 4 of {[}\ref{DSS}{]}, as well as Propositions 1 - 3 of that paper. 
As part of our concluding remarks in (3.3), we will discuss in more detail the status of the
omitted items.  

In (3.1.1), we introduce a number of notions related to the correspondence between values of the $\text{MN}^O_n$
and multinomial vectors $\vec{k}$; the latter arise as vectors (of length $m$) of frequencies of the outcomes, $o_s$
of $O$.  We also introduce the functions $\tau^O_1$ and $\tau^O_2$.  This is the substance of Definitions 2 - 5.  The 
role of $\tau^O_1$ is essential and has been mentioned in the Introduction, just prior to the beginning of (1.1).  On
the other hand, despite its formal dependence on $O$, the role of the function $\tau^O_2$ is merely a convenience
for the proof of the important Lemma 2.  This is discussed in somewhat more detail below.

In (3.1.2), Definition 6 introduces the function $\text{SMC}^O$ and the $\gamma^O_{n,t}$, which are this paper's
analogues of the function SBC of {[}\ref{DSS}{]} and the binomial coefficients, respectively.  In Definition 7, we introduce this paper's versions of
the functions $\text{Step}_n,\ \text{Weight}_n$.
Incorporating Remark 1 paves the way for introducing their ``integer versions'' $\text{IStep}_n,\ \text{IWeight}_n,\ \text{I}S\text{ and I}S^*$, respectively, following the paradigm laid out in the penultimate paragraph of (1.1).  We also
introduce the functions $\text{IStep},\ \text{IWeight}$ which are the natural encodings
(as functions of two variables) of the family of the 
$\text{IStep}_n,\ \text{IWeight}_n$, respectively.   Some obvious properties of these ``integer versions'' are
stated in Remark 2.

In Definition 8, we introduce the sets $A_{n,t}$ and $B_{n,t}$ (this paper's version of the sets $A_{n,i}$ and $B_{n,i}$
of {[}\ref{DSS}{]}), their ``integer versions'' $\text{I}A_{n,t}$ and
$\text{I}B_{n,t}$, the three-place relations  $\text{RI}A$ and
$\text{RI}B$, which encode the families $\text{I}A_{n,i}$
and $\text{I}B_{n,i}$ respectively, and, finallly, their cardinality functions, $\alpha (n,t,\xi )$ and $\beta (n,t,\xi )$,
respectively, and their enumerating functions $a_{n,t,s}$ and $b_{n,t,s}$, respectively.  The cardinality and
enumerating functions play an important role in the proofs of Lemma 2 and Theorem 2.  The notion of ``tameness''
which was introducted in {[}\ref{DSS}{]}, remains in the background here,  but plays an important role implicitly.  This will be
discussed in connection with the summary of (3.2).

The material of Definition 8 also sets the stage for the concluding items, Lemma 1 and Corollary 2, of (3.1.2).  In part (a)
of Lemma 1, we show that, in analogy with the situation of {[}\ref{DSS}{]}, the $\gamma^O_{n,t}$ are the common cardinalities,
$\alpha_{n,t}$ of $\text{I}A_{n,t}$ and $\beta_{n,t}$ of $\text{I}B_{n,t}$.  In part (b), we give an alternate characterization
of admissibility which enables part (c), where we finally tie up some loose ends from \S 2 by proving the existence of
many admissible permutations of $\left \{ 0, \ldots , 2^{n(M+1)}-1\right\}$.  In Corollary 2, we invoke Theorem 1 and Corollary 1 to extend this to the existence of representations (with the needed properties) of each $S^*_n$ and to strong, trim triangular array representations of the sequence $\left \{ S^*_n\right\}$.  We conclude (3.1) by developing, in (3.1.3), some notation related to the $\overline{J}_i$.
This will be used in the proof of Lemma 2.

We begin to address complexity issues in earnest in (3.2).  In Definition 9 ((3.2.1)), we introduce the sequence $\left\{ F_n\right\}$ of admissible permutations and note, in Remark 3, its connection to the view of admissibility developed in the final paragraph of the proof of Lemma 1.  This Remark justifies the somewhat imprecise claim that each $F_n$ is the {\it simplest } admissible permutation of 
$\left\{ 0, \ldots , 2^{n(M+1)}-1\right\}$.  

In (3.2.2) we prove a sequence of results that culminates in the proof of Theorem 2:  that
the two-place function $F$ which naturally encodes the sequence $\left\{ F_n\right\}$ is P-TIME relative to $\tau^O_1$.  
Before laying out in detail the route to this result, it will be helpful to comment briefly on the role of ``tameness'' and some of
the related results of {[}\ref{DSS}{]}.  The context is that we are dealing with a $d+1-\text{place relation, }R$ on $\mathbb{N}$ and
$R$ has the property that for all $\vec{u} = \left ( u_1, \ldots , u_d\right ) \in \mathbb{N}^d,$
\[
R[\vec{u}] = \left\{ w\vert R\left ( u_1, \ldots , u_d, w\right)\right\}\text{ {\it is finite}.}
\]
 The {\it cardinality function } for $R$ was then defined
to be the function, which, to $d+1-\text{tuples, }(\vec{u},\xi ) \in \mathbb{N}^{d+1}$, assigns the cardinality of 
$R[\vec{u}] \cap \{ 1, \ldots , \xi\}$.
Then, in {[}\ref{DSS}{]}, such a relation, $R$, was defined to be ``tame'' iff its cardinality function is P-TIME.  The important contribution
of tameness was provided by item (6) of Remark 3, and by Lemma 2 of {[}\ref{DSS}{]}, and both of these required a mild
additional assumption:  that it is P-TIME decidable whether, given $\vec{u} \in \mathbb{N}^d$, there is some $w$ such
that $R(\vec{u}, w)$.  Given this additional assumption on $R$, item (6) of Remark 3 of {[}\ref{DSS}{]} notes that
if $R$ is tame, then $R$ itself is P-TIME decidable.  This is simply because, for $\vec{u} \in \mathbb{N}^d$ and $w \in \mathbb{Z}^+$,  letting cd denote the cardinality function for $R$, $R(\vec{u},w)$ holds iff $\text{cd}(\vec{u},w) = \text{cd}(\vec{u},w-1) + 1$.

The content of Lemma 2 of {[}\ref{DSS}{]} involves the ``enumerating function'' for a relation $R$, as above.  This is the function, $E$, whose
domain consists of $d+1-\text{tuples. }(\vec{u},s)$ such that $1 \leq s \leq |R[\vec{u}|$, and to such a $d+1-\text{tuple}$
assigns the $s^{\text{th}}$ member of $R[\vec{u}]$ (in the increasing enumeration of $R[\vec{u}]$).  The result is that if
$R$ is tame, then $E$ is P-TIME. 

Returning to the context of this paper, of course the relations we have in mind are $\text{RI}A$ and $\text{RI}B$.  We
relativize all of the notions in the last two paragraphs to the functoin $\tau^O_1$.  Then, relative to $\tau^O_1$, both of
these relations satisfy the mild additional hypothesis:  their domains are P-TIME decidable relative to $\tau^O_1$.  
The proofs of item (6) of Remark 3, and of Lemma 2, both of {[}\ref{DSS}{]} then relativize in a completely straightforward
way to $\tau^O_1$.

We can now
continue with our account of the sequence of results in (3.2.2).  We
begin with Proposition 1 which establishes that the following functions are P-TIME relative to $\tau^O_1$:   $\text{SMC}^O$, the $\gamma^O_{n,t}$, the function IStep and the function $\alpha$ of Equation (7c) (Definition 8).  As already noted, above, $\alpha$ is the ``cardinality function'' for the relation $\text{RI}A$.   Therefore, in virtue of the preceding discussion, it follows (and this is Remark 4)
that  $\text{RI}A$ is {\it tame relative to }$\tau^O_1$, and so it is also P-TIME decidable relative to $\tau^O_1$.
This is part of Corollary 3 which also notes that for the same reason and since the function $a_{n,t,s}$ of Equation of (7di) (Definition 8) is the enumerating function for $\text{RI}A$, it follows that, as a function of $(n,t,s),\ a_{n,t,s}$ is P-TIME relative to $\tau^O_1$.

In Proposition 2 we show that $\tau^O_2$ is outright P-TIME, despite its formal dependence on $O$.  The main work is done in Lemma 2, where we show that the function $\beta$ of Equation (7c) is also P-TIME relative to $\tau^O_1$.  It follows (as above for $\text{RI}A$ and the function $a_{n,t,s}$) that $\text{RI}B$ is tame relative to $\tau^O_1$ and therefore P-TIME decidable relative to $\tau^O_1$,
and that the function $b_{n,t,s}$ of Equation (7dii) (Definition 8) is also P-TIME relative to $\tau^O_2$ (since it is the enumerating
function of $\text{RI}B$).  This is Corollary 4, the analogue ``on the $B-\text{side}$'' of Corollary 3.

This completes the preparation for the proof of Theorem 2.  We conclude (3.2) with the obvious translation
of the result of Theorem 2 into the language of strong trim triangular representations of the sequence $\left\{ S^*_n\right\}$.  This is Theorem 3 and it follows immediately from Theorem 2 and Corollary 1.  We conclude the paper in (3.3) with a retrospective comparison with {[}\ref{DSS}{]}, focusing on the status of the items of {[}\ref{DSS}{]} which we have omitted, and the underlying difficulties having to do with the fact that there is no longer a transparent connection between the values of $S_n$ and the number of 1's in the binary representation of $\ell < 2^{n(M+1)}$.
\subsection{Preliminaries for Theorems 2 and 3}
\label{subsec:3.1}
Our first task will be to take a closer look at the multinomial distributions $\text{MN}^O_n$
\subsubsection{Multinomial vectors, values and frequencies:  the functions $\tau^O_1$ and $\tau^O_2$}
\label{subsubsec:3.1.1}
\begin{definition}  We denote by $\mathbf{K}_n$ the set of level $n$ multinomial vectors,
$\vec{k} = \left ( k_1, \ldots , k_m\right )$, where each $k_i$ is a non-negative integer and $\sum_i k_i = n$.
As usual, $\binom{n}{\vec{k}}$ denotes the multinomial coefficient corresponding to $n$ and $\vec{k}$, i.e.
$\binom{n}{\vec{k}} = \frac{n!}{k_1!\cdots k_m!}$.
We denote by $C_n$ the cumulative
distribution function of $\text{MN}^O_n$.
\end{definition}

As already noted in (1.1), the values of $\text{MN}^O_n$ are the inner products, 
$\vec{k}\cdot\vec{o} = k_1o_1 + \ldots + k_mo_m$, of the $\vec{k} \in \mathbf{K}_n$ with the fixed vector,
$\left ( o_1, \ldots , o_m\right )$, of outcomes of $O$.    This leads naturally to the following
equivalence relations, $\equiv_n$ on the  $\mathbf{K}_n$.

\begin{definition}  For $\vec{k},\ \vec{\kappa} \in \mathbf{K}_n$, we set $\vec{k}
\equiv_n \vec{\kappa}$ iff $\vec{k}\cdot\vec{o} = \vec{\kappa}\cdot\vec{o}$.
\end{definition}

Thus, it is the equivalence classes of $\equiv_n$ that correspond bijectively to the values of 
$\text{MN}^O_n$.  The structure of the $\equiv_n$ depends on $O$, and therefore, ultimately,
on $Q$ (recall (1.2)) and thus we cannot control this structure aside from the minimal requirements
imposed by the hypotheses that $Q$ has been normalized.  In particular, we cannot control the
number of values/equivalence classes beyond the obvious bound of $\binom{n+m-1}{m-1}$  
mentioned in (1.1) (just prior to Equation (1)).  

\begin{definition}  We let $\left ( v_n^t\vert 0 \leq t \leq T_n\right )$ be the increasing enumeration of
the values of $\text{MN}^O_n$.  For $0 \leq t \leq T_n$, we let $\mathbf{K}_{n,t}$ be the equivalence class of  $\equiv_n$
corresponding to the value $v^t_n$, i.e., $\mathbf{K}_{n,t} := \left\{ \vec{k} \in \mathbf{K}_n\vert 
\vec{k}\cdot\vec{o} = v^t_n\right \}$ and we let $\mathbf{\overline{K}}_{n,t} := \left\{ \vec{k} \in \mathbf{K}_n\vert 
\vec{k}\cdot\vec{o} < v^t_n\right \}$.  

We let $\tau^O_1$ be that function whose domain consists of all $m+2-\text{tuples},\ \left ( n,k_1,\ldots , k_m,t\right)$
such that $ n > 0,\ 0 \leq t \leq T_n,\ \left ( k_1,\ldots , k_m\right ) \in \mathbf{K}_n$, and such that:
\[
\tau^O_1 \left ( n, k_1, \ldots , k_m, t \right ) = 1\text{ iff }\left ( k_1, \ldots , k_m \right ) \in \mathbf{K}_{n,t}.
\]
We will also denote elements of the domain of $\tau^O_1$ by $( n,\vec{k},t)$.
\end{definition}
The $t-\text{indices}$ of these
values will serve as the analogues of the notions of Step and Weight from {[}\ref{DSS}{]}.
The analogues in {[}\ref{DSS}{]} of the $v_n^t$ are the $-n + 2i$, for $0 \leq i \leq n$, while
the analogue of $\tau^O_1$ in {[}\ref{DSS}{]} is simply the (obviously P-TIME) function which (for $n, k$, and $t$, with $0 \leq i \leq n$ and $-n \leq t \leq n$ such that $n \equiv t (\text{ mod } 2)$) takes on value 1 iff $t = -n+2i$.  Thus, it checks whether or not $t$ is the
value of $S_n$ that corresponds to the Hamming weight, $k$.  Our $\tau^O_1$ furnishes the analogous check.  
As usual, in our setting, the correspondence between patterns of binary digits and values of $S_n$ is no longer at the surface: it depends on $O$ in an essential way.
\begin{definition}
For $x \in (0,1),\ 0 \leq b \leq n$ and $1 \leq s \leq m$, let $k_{n,s}(x,b) := $ the
frequency of $o_s$ in $\left ( \left (\vec{o}_n(x)\right)_i\vert i > b\right )$ and let
$\vec{k}_n(x, b) := 
\left ( k_{n,1}(x,b), \ldots , k_{n,m}(x,b)\right )$.  Since these are
constant on the $E_{n, \ell}$, we  have also defined the integer versions, $\text{I}k_{n,s}(\ell ,b)$
and $\text{I}\vec{k}_n(\ell ,b)$, for $0 \leq \ell < 2^{n(M+1)}$.  We let $\tau^O_2$ be the function defined on
four-tuples $(n,s,\ell , b)$ with $0 \leq b \leq n,\ 1 \leq s \leq m$ and $0 \leq \ell < 2^{n(M+1)}$, 
and such that for such $(n,s,\ell ,b)$:
\[
\tau^O_2(n,s,\ell ,b) := \text{I}k_{n,s}(\ell ,b).
\]
\end{definition}
\subsubsection{Multinomial analogues of \textnormal{Step, Weight} and related notions}
\label{subsubsec:3.1.2}
We now develop the analogues of the notions introduced in (3.1) of {[}\ref{DSS}{]}.
\begin{definition}   For $n >  0$ and $0 \leq t \leq T_n+1$, set 
\[
\text{SMC}^O(n,t) := \sum_{\vec{k} \in \mathbf{\overline{K}}_{n,t}}\ \binom{n}{\vec{k}}\text{ and } \gamma^O_{n,t} :=
\text{SMC}^O(n,t+1) - \text{SMC}^O(n,t).
\]
\end{definition}
\noindent
Thus, $\displaystyle{\gamma^O_{n,t} = \sum_{\vec{k} \in \mathbf{K}_{n,t}}\binom{n}{\vec{k}}}$.  As we shall soon see (and as may already be obvious), $\gamma^O_{n,t}$ is the combinatorial counterpart of the binomial coefficient $\binom{n}{i}$ in {[}\ref{DSS}{]}, while $\text{SMC}^O$ is the combinatorial counterpart of the function $\text{SBC}$ of {[}\ref{DSS}{]}.  This said, even these ``aggregated'' functions 
fall short of providing all of the information needed for the proof of the crucial Lemma 2, below, in (3.2).  

For $\vec{k} \in
\mathbf{K}_n$ and $0 \leq t \leq T_n,\ \tau^O_1$ tells whether or not $\vec{k} \in \mathbf{K}_{n,t}$.   Item (2) of Remark 1, and the related Proposition 1, below, show that the information encoded by $\tau^O_1$ allows us to recover the $\gamma^O_{n,t}$ and $\text{SMC}^O$ in a simple fashion.  The function $\tau^O_2$ provides information about the frequencies of the outcomes $o_s$ that is ``stratified'' by $b$, i.e., about the frequencies among the $\left ( \text{I}\vec{o}_n(\ell )\right )_i$ for $i > b$.  

Both $\tau^O_1$ and $\tau^O_2$ play prominent roles in the proof of Lemma 2, but the role of $\tau^O_2$ will turn
out to be a convenience, since, despite its formal dependence on $O$, we will nevertheless be able to show, in Proposition 2,
of (3.2), that it is outright P-TIME.   The finer information encoded in $\tau^O_1$ is indispensable, however, and thus, in (3.2),
our upper complexity bounds will be obtained relative to the function $\tau^O_1$.  
\begin{definition}
For $n > 0$, and $x \in (0,1)$,
$\text{Step}_n(x)$ (respectively $\text{Weight}_n(x)$)  is the unique $t$ with $0 \leq t \leq T_n$
such that $S^*_n(x) = v_n^t$ (respectively $S_n(x) = v_n^t$).
\end{definition}
\begin{remark}  For $x \in (0,1)$. 
and $n > 0$, the following observations are obvious:
\begin{enumerate}

\item  $\text{SMC}^O(n,t) = 
2^{n(M+1)}C_n\left ( v^t_n\right ),$

\item  $\displaystyle{\gamma^O_{n,t} = 
\underset{\vec{k} \in \mathbf{K}_n}\sum\tau^O_1(n,\vec{k},t)}$ so
$\displaystyle{\text{SMC}^O(n,t) =
\underset{t' < t}\sum\underset{\vec{k} \in \mathbf{K}_n}\sum\tau^O_1(n,\vec{k},t')}$,

\item  $\text{Step}_n(x)$ is the unique $t$ such that $\text{SMC}^O(n,t) \leq x2^{n(M+1)} < 
\text{SMC}^O(n,t+1)$,

\item  $\text{Step}_n(x),\ \text{Weight}_n(x)$ depend at most on $\vec{b}_n(x),\
\vec{o}_n(x)$, respectively.\qed
\end{enumerate}
\end{remark}
In view of the penultimate paragraph of (1.1) and 3. of Remark 1, we can introduce the notions $\text{IStep}_n(\ell ),\
\text{IWeight}_n(\ell )$ for $0 \leq \ell < 2^{n(M+1)}$.  We already knew that $S_n$ and $S^*_n$ also depend at most
on $\vec{b}_n(x), \vec{o}_n(x)$, respectively, and the $\text{I}S^*_n(\ell ),\ \text{I}S_n(\ell )$ for
$0 \leq \ell < 2^{n(M+1)}$ were already introduced at the end of (1.1).  Via the usual identification, we have also 
introduced the notions  
$\text{I}S^*_n(\vec{b}),\ \text{I}S_n(\vec{o})$ for
$\vec{b} \in \{ 0, 1\}^{n(M+1)}$ and $\vec{o} \in \mathbf{O}_n$.  
Unlike the situation in {[}\ref{DSS}{]}, here, due to the role of $O$, $\text{IWeight}_n(\ell )$ will typically {\it not}
be independent of $n$ and purely intrinsic to $\ell$; there is greater symmetry between Step and Weight
in the present context.
{\it In what follows, we shall use the notations }$\text{IStep}(n,\ell ),\
\text{IWeight}(n,\ell )$
rather than $\text{IStep}_n(\ell ),\ \text{IWeight}_n(\ell )$, respectively.  
The following is then also obvious. 
\begin{remark}  For $n > 0$ and $k < 2^n$:
\begin{enumerate}
\item  $\text{IStep}(n,0) = 0$ and for $0 < \ell < 2^{n(M+1)} ,\ 
\text{IStep}(n,\ell )$ is the least positive $t \leq T_n$ such that 
$\displaystyle{\ell < \text{SMC}^O(n,t+1)}$,

\item  For all $0 \leq t \leq T_n,\
\text{SMC}^O(n,t)$ is the unique (and so least) $\ell$ such that $\text{IStep}(n,\ell ) = t.$\qed
\end{enumerate}
\end{remark} 

\begin{definition}  For   
$0 \leq t \leq T_n$, we define $A_{n,t},\ B_{n,t}$ by:
\begin{equation}
A_{n,t} :=
\{ x \in (0,1)\vert \text{Step}_n(x)\ = t\}\text { and }  
B_{n,t} := \{ x \in (0,1)\vert \text{Weight}_n(x)\ = t\}.
\end{equation}
In view of Remark 2, for fixed $n > 0$, each $A_{n,t}$
is a union of $D_{n,\ell }$'s and each $B_{n,t}$ is the 
union of $E_{n,\ell }$'s, and therefore,
in view of the last paragraph of (1.1), we have introduced the notions $\text{I}A_{n,t},\ \text{I}B_{n,t}$ 
for the corresponding subsets of $\left\{ 0, \ldots , 2^{n(M+1)}-1\right\}$\ 
(or of $\{ 0, 1\}^{n(M+1)}$, via the usual identification) :\
\begin{equation}\tag{7a}
\text{I}A_{n,t} := \left\{ \ell < 2^{n(M+1)}\vert D_{n,\ell } \subseteq A_{n,t}\right\}\text{ and }
\text{I}B_{n,t} := \left\{ \ell < 2^{n(M+1)}\vert E_{n,\ell } \subseteq B_{n,t}\right\}.
\end{equation}
We also let:
\begin{equation}\tag{7b}
\alpha_{n,t} :=  \left\vert\text{I}A_{n,t}\right\vert\text{ and }
\beta_{n,t} :=  \left\vert\text{I}B_{n,t}\right\vert,
\end{equation}
and for positive integers, $\xi < 2^{n(M+1)}$,
we let 
\begin{equation}\tag{7c}
\alpha (n,t,\xi) =  \left\vert\text{I}A_{n,t} \cap \{1,\ldots , \xi\}\right\vert,\text{ and }
\beta (n,t,\xi) =  \left\vert\text{I}B_{n,t} \cap \{1,\ldots , \xi\}  \right\vert. 
\end{equation}
For $n > 0$ and $1 \leq t \leq T_n$ we let:
\begin{equation}\tag{7di}
\left ( a_{n,t,s}\vert 1 \leq s \leq \alpha_{n,t}\right )\text{  be the increasing enumeration of }
\text{I}A_{n,t},
\end{equation}
\begin{equation}\tag{7dii}
\left ( b_{n,t,s}\vert 1 \leq s \leq \beta_{n,t}\right )\text{  be the increasing enumeration of }
\text{I}B_{n,t}.
\end{equation}
Finally, we let:
\begin{equation}\tag{7e}
\text{RI}A(n,t,\ell )\text{ iff } \ell \in \text{I}A_{n,t}\text{ and }
\text{RI}B(n,t,\ell )\text{ iff } \ell \in \text{I}B_{n,t}.
\end{equation}
\end{definition}

In the terminology of {[}\ref{DSS}{]}, the functions of Equation (7c) are the cardinality functions for the relations of Equation (7e),
and the functions of Equations (7di) and (7dii) are the enumerating functions for these relations.

\begin{lemma}  Let $n > 0$.  Then:
\begin{enumerate}
\item  for all $0 \leq t \leq T_n,\ \alpha_{n,t} = \gamma^O_{n,t} = \beta_{n,t}$.
\item  If $\pi$ is a permutation of 
$\left\{ 0, \ldots , 2^{n(M+1)}-1\right\}$,
the admissibility of $\pi$  is equivalent to each of the 
following conditions:
\begin{enumerate} 
\item  for all $\ell < 2^{n(M+1)},\
\text{\textnormal{IWeight}}(n,\pi (\ell )) = \text{\textnormal{IStep}}(n, \ell )$,
\item  for all $0 \leq t \leq T_n,\ \pi\left [ \text{\textnormal{I}}A_{n,t}\right ] =
\text{\textnormal{I}}B_{n,t}$,
\item  $\text{\textnormal{I}}S^*_n = \text{\textnormal{I}}S_n \circ \pi$.
\end{enumerate}  
\item  There are $\displaystyle{\prod_{i=0}^{T_n} \left (\gamma^O_{n,t}!\right )}$
admissible permutations of $\left\{ 0, \ldots , 2^{n(M+1)}-1\right \}.$ 
\end{enumerate}
\end{lemma}
\begin{proof}  Since $\alpha_{n,t} = 2^{n(M+1)}P\left ( \text{Step}_n(x) = v^t_n\right )$ and 
$\beta_{n,t} = 2^{n(M+1)}P\left ( \text{Weight}_n(x) = v^t_n\right )$,\\ 1. follows.   For 2., it is clear that (c) is equivalent to (a), and from (1) it then follows that (b) is also equivalent to (a), so we argue that the admissibility of $\pi$ is equivalent to (a).  
Let $\pi$
be any permutation of $\left\{ 0, \ldots , 2^n - 1\right\}$, let $\ell < 2^{n(M+1)}$ and 
let $x \in D_{n,\ell },\ y \in E_{n, \pi (\ell )}$.  Let $s = \text{IStep}(n,\ell )$ and let $t = \text{IWeight}(n, \pi (\ell ))$.
Then $s = t$ iff $v^s_n = v^t_n$ iff $S^*_n(x) = S_n(y)$ and (a) is equivalent to the admissibility of $\pi$.

For 3., note that an admissible permutation $\pi$ 
decomposes into the system of its restrictions to the $\text{I}A_{n,t}$.
Complete information 
about $\pi\upharpoonright \text{I}A_{n,t}$ is encoded by
the permutation, $\overline{\pi}_{n,t}$ of
$\left\{ 1, \ldots , \gamma^O_{n,t}\right\}$ defined by:
\begin{equation}
\text{if } 1 \leq s \leq \gamma^O_{n,t},\text{ then }
\pi \left ( a_{n,t,s}\right ) = b_{n,t,\overline{\pi}_{n,t}(s)}.
\end{equation}
Further, the $\overline{\pi}_{n,t}$
are arbitrary in the sense that if, for $0 \leq t \leq T_n,\
\phi_{n,t}$ is {\it any } permutation of
$\left\{ 1, \ldots , \gamma^O_{n,t}\right\}$, then for each $n$
there is a (unique) admissible permutation $\pi_n$
of $\left\{ 0, \ldots , 2^{n(M+1)} - 1\right\}$ such
that for each $0 \leq t \leq T_n,\ \overline{\pi}_{n,t} = \phi_{n,t}$.  Finally,
for fixed $n$, the product in 3. counts the number of such
systems $\left ( \phi_{n,t}\vert 0 \leq t \leq T_n\right )$,
and so 3. follows.
\end{proof}

The next Corollary is an immediate consequence of Theorem 1, Corollary 1 and item
3. of Lemma 1; it gives the existence of strong trim triangular arrays for $\left\{ S^*_n\right\}$.
\begin{corollary}  For each $n$,
there are 
$\displaystyle{\prod_{i=0}^{T_n} \left ( \gamma^O_{n,t}!\right )\text{ representations of } S^*_n\text{ as a sum, }
S^*_n =  \sum_{i=1}^nR^*_{n,i}}$,
where $\left ( R^*_{n,i}\vert\, 1 \leq i \leq n\right )$ is an
i.i.d. family of random variables each of which 
has outcomes of $O$
as values, with equal probability,
has mean 0, variance 1 and depends 
only on $\vec{b}_n.$   Therefore, there exist (continuum many)
strong trim triangular array representations of the sequence
$\left\{ S^*_n\right\}$.
\qed
\end{corollary}
\subsubsection{A Closer Look at the $\overline{J}_i$}
\label{subsubsec:3.1.3}

We develop here some additional properties of the the $\overline{J}_i$
needed for the proof of Theorem 2, and especially of the crucial Lemma 2,
both in the next subsection.  It is there that we will see the point of the particular
choice of the triangular array of positive integers presented in (1.1) and its apparently peculiar
``repetition'' ($M+1$ times) of each row length.

With $\binom{1}{2} = 0$, as usual, it is easy to see that the integer entries in the $b^{\text{th}}$ block of rows are
precisely those $\eta$ satifsying $\binom{b}{2}(M+1) < \eta \leq \binom{b+1}{2}(M+1)$, and therefore, that the integer
entries in the $r^{\text{th}}$ row of the $b^{\text{th}}$ block are precisely those $\eta$ satisfying 
$\left ( \binom{b}{2} + (r-1)\right )(M+1) < \eta \leq \left ( \binom{b}{2} + r\right )(M+1)$.  For $1 \leq i \leq b,\ 
1 \leq r \leq M+1$, we let $\eta_{b,r,i}$ be the entry in row $r$ and column $i$ of block $b$.  It is then
immediate that $\eta_{b,r,i} =  \left ( \binom{b}{2} + (r-1)\right )(M+1) + i$.

For fixed $b$ and fixed $1 \leq i \leq b$, the entries in the $i^{\text{th}}$ column of the $b^{\text{th}}$
block are precisely the $\eta_{b,r,i}$, with $1 \leq r \leq M+1$, i.e., the $\left ( \binom{b}{2} + (r-1)\right )(M+1) + i$.
These are threfore exactly the members of the intersection of $J_i$ with the $b^{\text{th}}$ block of rows.
Finally, we have that the elements of the $\overline{J}_i$ are the entries in the $i^{\text{th}}$ (and last)
column in the $i^{\text{th}}$ block of rows; i.e., 
$\overline{J}_i$ is the intersection of the $i^{\text{th}}$ column with the $i^{\text{th}}$
block of rows, and that for $1 \leq r \leq M+1,\ j_{i,r} = \eta_{i,r,i} = \left ( \binom{b}{2} + (r-1)\right )(M+1) + i$.
\subsection{Complexity Estimates: Theorems 2 and 3}
\label{subsec:3.2}

\subsubsection{The sequence $\left\{F_n\right\}$, its natural encoding, $F$, and Theorem 2}
\label{subsubsec:3.2.1}
For the next Definition, recall
Definition 8 in (3.1), where we defined 
the $b_{n,t,s}$ and $\alpha (n,t,x)$.
\begin{definition}  For all $n > 0,\ F_n$ is the permutation
of $\left\{ 0, \ldots , 2^{n(M+1)}-1\right \}$ defined as follows.
If $0 \leq \ell < 2^{n(M+1)}$, let $t = \text{IStep}(n,\ell )$.  Then
\[
F_n(\ell ) := b_{n,t,s},\text{ where } 
 s = \alpha (n,t,\ell ).
\]
Then,  take 
$F:  \mathbb{N}\times \mathbb{N}
\to \mathbb{N}$ to be the natural encoding of the sequence
$\left\{ F_n\right\}$.  
\end{definition}
\noindent
In view of 2. of Lemma 1, these $F_n$ are (obviously)
very natural admissible permutations of the $\left\{ 0, \ldots , 2^{n(M+1)}-1\right \}$.
\begin{remark}  Note that our definition of $F_n$ is equivalent
to stipulating that, in terms of the notation of Equation (8),
for all $1 \leq t \leq T_n,\ \left ( \overline{F_n}\right )_t$ is the identity permutation of
$\left\{ 1, \ldots , \gamma^O_{n,t}\right\}$.  This is one aspect of what, in our eyes,
justifies the claim that $F_n$ is the simplest admissible permutation of 
$\left\{ 0, \ldots , 2^{n(M+1)} - 1\right \}$.
Note, also, that with $t = \text{\textnormal{IStep}}(n,\ell )$ and $s =
\alpha(n,t,\ell )$, then, in fact, $s = 1 + \ell - \text{\textnormal{SMC}}^O(n,t)$; further,
for these $t, s$, we also have that $s = \beta (n,t,F(n,\ell ))$.
\qed
\end{remark} 
\subsubsection{$F$ is \textnormal{P-TIME} relative to $\tau^O_1$}
\label{subsubsec:3.2.2}
It was noted in Proposition 4 of {[}\ref{DSS}{]} that, as 
functions of $(n,i)$, with $i \leq n$, the
binomial coefficients $\binom{n}{i}$ and \textnormal{SBC} are 
computable in time polynomial in $n$.  Building on this, we have:

\begin{proposition}   As functions of $(n,t)$, with $1 \leq t \leq T_n$, the
function $\text{SMC}^O$ and the function $\gamma^O_{n,t}$ are computable in time
polynomial in $n$, relative to the function $\tau^O_1$.
The function $\text{\textnormal{IStep}}$ and the function $\alpha$ of \textnormal{Equation (7c) (Definition 8)} are \textnormal{P-TIME}
relative to $\tau^O_1$.  
\end{proposition}
\begin{proof}  First note that, as a function of $\vec{k} \in \mathbf{K}_n$ and of  $n ,\
\binom{n}{\vec{k}}$ is computable in time polynomial in $n$, by Proposition 4 of {[}\ref{DSS}{]} and
the standard expression for the multinomial coefficient as a $m-\text{fold}$ product
of binomial coefficients.  The first sentence of the Proposition is then immediate from item (2) of
Remark 1.  

That $\text{IStep}$ is P-TIME relative to $\tau^O_1$ then follows immediately from item (1) of Remark 2, and
that $\alpha$ is P-TIME relative to $\tau^O_1$ follows immediately from (the first part of) the final sentence
of Remark 3.
\end{proof}

\begin{remark}  In (a generalization of) the terminology of {[}\ref{DSS}{]}, we have shown that the relation
$\text{\textnormal{RI}}A$ is {\it tame} relative to $\tau^O_1$.  This is because, as noted, immediately after Definition 8, 
in the terminology of {[}\ref{DSS}{]}, $\alpha$ is the cardinality function for $\text{\textnormal{RI}}A$.
\end{remark}

The next Corollary then follows immediately from Proposition 1 and from Lemma 2 of {[}\ref{DSS}{]}.
\begin{corollary}  
The relation $\text{\textnormal{RI}}A$ is \textnormal{P-TIME} decidable relative to $\tau^O_1$.
As a function of $(n,t,s)$ with $1 \leq s \leq \alpha_{n,t}$, the function $a_{n,t,s}$ of \textnormal{Equation (7di) (Definition 8)} is \textnormal{P-TIME} relative to $\tau^O_1$.
\end{corollary}
\begin{proof}  This is simply because, as noted following Definition 8, the function $a_{n,t,s}$ is
the enumerating function for $\text{\textnormal{RI}}A$.
\end{proof}

We next show, as promised, that the dependence of $\tau^O_2$ on $O$ is really only formal.  It is now that we make
use of the enumeration, $\left ( \varsigma_s\vert 1 \leq s \leq m\right )$ of $\{ 0, 1\}^m$ introduced in the third
paragraph of (1.1).
\begin{proposition}  As a function of $(n, s, \ell, b)$ with $n > 0,\ 1 \leq s \leq m,\  
0 \leq \ell \leq 2^{n(M+1)},\ 0 \leq b\leq n,\ \tau^O_2$ is computable in time polynomial in $n$.
\end{proposition}
\begin{proof}   $\tau^O_2(n, s, \ell, b)$ is the sum over $i$ such that $b < i \leq n$ of the characteristic (indicator) function of the relation,
$\Theta (n, s,\ell , i)$ which holds iff $\left ( \text{I}\vec{o}_n(\ell )\right )_i = o_s$.  This relation is P-TIME decidable since
it holds iff for all $j$ with $1 \leq j \leq M+1,\ \left ( \text{I}\vec{b}_n(\ell )\right )_j = \left ( \varsigma_s\right )_j$.
\end{proof}
We now have the analogue of Lemma 3 of {[}\ref{DSS}{]}.  The argument here is considerably more
involved and delicate, and requires the information encoded in $\tau^O_1$.  As already
noted, appeals to $\tau^O_2$ are just a convenience, in view of Proposition 2 .
Nevertheless, as with its prototype in {[}\ref{DSS}{]}, a key ingredient is S. Buss's suggestion of
(something like) a binary search.  
\begin{lemma}
The function $\beta (n,t,\xi)$ of 
Definition 8 is \textnormal{P-TIME} relative to $\tau^O_1$.
\end{lemma}
\begin{proof}   In analogy with the proof of Lemma 3 of {[}\ref{DSS}{]}, given $n, t$, and $0 \leq x < 2^{n(M+1)}$, we ``walk down the 1's of the binary representation of $\xi$''; here we use the multinomial coefficients to count as we go.  More precisely, with 
\[
\xi = \sum_{\zeta =1}^{n(M+1)} 2^{n(M+1)-\zeta}\left ( \text{I}\vec{b}_n(\xi )\right )_\zeta ,\text{  (as in (2.1.3) with }\xi\text{ in place of }x),
\]
if $\left ( \text{I}\vec{b}_n(\xi )\right )_\zeta = 0$, we let $\text{I}\overline{B}_{n,t,\xi ,\zeta } := \emptyset$, while if
$\left ( \text{I}\vec{b}_n(\xi )\right )_\zeta = 1$, we let 
\[
\text{I}\overline{B}_{n,t,\xi ,\zeta } := \left\{ \ell \in \text{I}B_{n,t}\biggr\vert \left ( \text{I}\vec{b}_n(\ell )\right )_\zeta = 0\text{ and for all }j < \zeta ,\
\left ( \text{I}\vec{b}_n(\ell )\right )_j = \left ( \text{I}\vec{b}_n(\xi )\right )_j\right\}.
\]
If $\left ( \text{I}\vec{b}_n(\xi )\right )_\zeta = 1,\ \xi \in \text{I}B_{n,t}$ and for all $j > \zeta,\
\left ( \text{I}\vec{b}_n(\xi )\right )_j = 0$, we let $\text{I}B_{n,t,\xi ,\zeta } := \text{I}\overline{B}_{n,t,\xi ,\zeta } \cup \{\xi\}$;
otherwise, we let $\text{I}B_{n,t,\xi ,\zeta } := \text{I}\overline{B}_{n,t,\xi ,\zeta }$. 
Finally, we set:  $\beta^*(n,t,\xi ,\zeta ) := \left\vert \text{I}B_{n,t,\xi ,\zeta }\right\vert$.

At ``stage $\zeta$'', in the non-trivial case, where $\left ( \text{I}\vec{b}_n(\xi )\right )_\zeta = 1$, we use the multinomial coefficients to compute $\beta^*(n,t,\xi ,\zeta )$.  Since (with $n, t, \xi$
fixed) the $\text{I}B_{n,t,\xi ,\zeta }$ are pairwise disjoint and their union is $\text{I}B_{n,t} \cap \{ 1, \ldots , \xi \}$, we'll have that:
\[ 
\beta (n,t,\xi ) = \sum_{\zeta =1}^{n(M+1)}\beta^*(n,t,\xi ,\zeta ).
\]

Thus, it suffices to show that $\beta^*$ is P-TIME relative to $\tau^O_1$, which we now undertake.  When $\left ( \text{I}\vec{b}_n(\xi )\right )_\zeta = 1$, we consider the superset, $\overline{C}(n,\xi ,\zeta ),$ of 
$\text{I}\overline{B}_{n,t,\xi ,\zeta }$ consisting of those
$\ell$ with $1 \leq \ell < \xi$ such that $\left ( \text{I}\vec{b}_n(\ell )\right )_\zeta = 0$ and for all $j < \zeta ,\
\left ( \text{I}\vec{b}_n(\ell )\right )_j = \left ( \text{I}\vec{b}_n(\xi )\right )_j$.  If $\xi \in 
\text{I}B_{n,t,\xi ,\zeta }$, we let $C(n,\xi ,\zeta ) := \overline{C}(n,\xi ,\zeta ) \cup \{\xi\}$; otherwise, we let
$C(n,\xi ,\zeta ) := \overline{C}(n,\xi ,\zeta )$.  

For $\ell \in C(n,\xi ,\zeta ),\ \ell \in
\text{I}B_{n,t,\xi ,\zeta }$ iff $\text{I}\vec{k}_n(\ell ) \in \mathbf{K}_{n,t}$ 
iff $\tau^O_1\left ( n, \text{I}k_{n,1}(\ell ), \ldots , \text{I}k_{n,m}(\ell ), t\right ) = 1$ (note that if
$\ell = \xi$, then this will automatically be true, by the definition of $C(n,\xi,\zeta )$).  

For $n, \xi , \zeta$ as above, and for $\vec{k} \in \mathbf{K}_n$, let:
\[
\overline{c}(n,\xi ,\zeta ,\vec{k}) := \text{the frequency of }\vec{k}\text{ in }\left ( \text{I}\vec{k}_n(\ell )\vert \ell \in \overline{C}(n,\xi ,\zeta )\right ),\text{ and }
\]
\[
c(n,\xi ,\zeta ,\vec{k}) := \text{the frequency of }\vec{k}\text{ in }\left ( \text{I}\vec{k}_n(\ell )\vert \ell \in C(n,\xi ,\zeta )\right ).
\]
It is then immediate that:
\begin{equation}
\beta^*(n,t,\xi ,\zeta ) = \underset{\vec{k} \in \mathbf{K}_n}\sum
c(n,\xi ,\zeta ,\vec{k})\tau^O_1(n,\vec{k},t).
\end{equation} 

It is also clear that $c(n,\xi ,\zeta ,\vec{k}) = \overline{c}(n,\xi ,\zeta ,\vec{k})$, 
unless $\xi \in C(n,\xi ,\zeta )$ and $\text{I}\vec{k}_n(\xi ) = \vec{k}$; in this case
$c(n,\xi ,\zeta ,\vec{k}) = \overline{c}(n,\xi ,\zeta ,\vec{k}) + 1$.  
It will, therefore, suffice to show that the function $\overline{c}$ is P-TIME, which is the burden of what follows.

Let $b(\zeta ),\ r(\zeta ),\ i(\zeta )$ be those $b, r, i$, respectively, such that, as in (3.1.3), $\zeta = \eta_{b,r,i}$.   Clearly, these are P-TIME computable from $\zeta$.  Note that if $\ell \in \overline{C}(n,\xi ,\zeta )$ and $i > b(\zeta )$, 
then $\left ( \text{I}\vec{o}_n(\ell )\right )_i = \left ( \text{I}\vec{o}_n(\xi )\right )_i$.  
Let $s(n, \zeta , \ell )$ be that $s$ such that $\left ( \text{I}\vec{o}_n(\ell )\right )_{\zeta} = o_s$, and note that $s(n, \zeta , \ell )$ is also that $s$ such that
$\text{I}k_{n,s}(\ell, b(\zeta )) < \text{I}k_{n,s}(\ell, b(\zeta )-1)$, i.e., such that $\tau^O_2(n,s,\ell ,b(\zeta )) <
\tau^O_2(n,s,\ell ,b(\zeta )-1)$; thus, $s(n, \zeta, \ell )$ is polynomial time computable. 

Let $D := \overline{J}_{b(\zeta )} \cap \{ 1, \ldots , \zeta-1\}$, and let $d = \left\vert D\right\vert$.   For $\ell \in 
\overline{C}(n,\xi ,\zeta )$, let $\sigma_\ell$ be that function $\sigma$ from $D$ to $\{0, 1\}$ defined
by $\sigma (j) = \left ( \text{I}\vec{b}_n(\ell )\right )_j$.   Clearly
$s(n,\zeta , \ell )$ depends only on $\sigma_\ell$, and equally clearly, every $\sigma \in \{0, 1\}^D$ arises
as $\sigma_\ell$, for some $\ell \in \overline{C}(n,\xi ,\zeta )$.  A key step in what follows will be to compute the
frequency with which this occurs, but for now, this justifies defining $s^*:\{0, 1\}^D \to \{ 1, \ldots , m\}$ by 

\begin{equation}
s^*(\sigma ) = s(n,\zeta , \ell)\text{ for any }\ell \in \overline{C}(n,\xi ,\zeta )\text{ such that }\sigma = \sigma_\ell.
\end{equation} 

We are now in a position to prove:
\begin{equation}
\vert \overline{C}(n,\xi ,\zeta )\vert = 2^{(b(\zeta )-1)(M+1)+d}.
\end{equation}

\begin{proof}{\it (of Equation (11))}  We give an explicit bijection $h$ from $\overline{C}(n,\xi ,\zeta )$ to the set $C^*$ of pairs,
$\left (\ell^*, \sigma\right )$ with $0 \leq \ell^* < 2^{(b(\zeta )-1)(M+1)}$ and $\sigma \in \{ 0, 1\}^D$.  
Given $\ell \in \overline{C}(n,\xi ,\zeta )$,
we let $h_1(\ell )$ be that $\ell^*$ such that $E_{n,\ell} \subseteq  E_{b(\zeta )-1,\ell^*}$, and let
$h_2(\ell ) = \sigma_\ell$.  Let $h(\ell ) :=
\left ( h_1(\ell ), h_2(\ell )\right )$.   Then, $h$ is one-to-one, since if $\ell,\ \ell ' \in \overline{C}(n,\xi ,\zeta )$ and $\ell ' \neq \ell$,
choosing $y \in E_{n, \ell }$, then for some $j \in D \cup \bigcup\left\{ \overline{J}_i\vert 1 \leq i < b(\zeta )\right\}$ and all $y' \in E_{n,\ell '},\ \varepsilon_j\left ( y'\right ) \neq \varepsilon_j(y)$.  Fixing such a $j$, if $j \in D$, then $h_2\left ( \ell  '\right ) \neq h_2(\ell )$, while if $j \not\in D$, then $h_1\left ( \ell '\right ) \neq h_1(\ell )$.  But $h$ is also onto, since if $0 \leq \ell^* < 2^{(b(\zeta )-1)(M+1)}$ and $\sigma : D \to \{ 0, 1\}$, let $y^* \in E_{b(\zeta )-1, \ell^*}$, let $y \in E_{n,\xi}$.  We may assume, WLOG, that $y^*$ has
the following properties, since, if necessary, it can be modified to have them, 
without affecting membership in $E_{b(\zeta )-1, \ell^*}$: 
\begin{itemize}
\item $\varepsilon_j\left ( y^*\right ) = \varepsilon_j(y)$, for all $j > \zeta$,
\item $\varepsilon_\zeta\left ( y^*\right ) = 0$,
\item $\varepsilon_j\left ( y^*\right ) = \sigma (j)$, for all $j \in D$.
\end{itemize}
Let $\ell$ be such that $y^* \in E_{n,\ell}$.  By construction, $\ell \in \overline{C}(n,\xi ,\zeta )$ and 
$h(\ell ) = \left ( \ell^*,  \sigma\right )$.
\end{proof}

We will continue to work with the bijection $h$ constructed in the proof of Equation (11).   
Let $\ell \in \overline{C}(n,\xi ,\zeta )$, and
let $h(\ell ) = \left ( \ell^*, \sigma\right )$.  It is then clear from the construction of $h$ 
that for all $1 \leq s \leq m$: 
\begin{equation}
\text{I}k_{b(\zeta )-1,s}\left ( \ell^*, 0\right ) = 
      \begin{cases}
                           \tau^O_2(n,s,\ell ,0) - \tau^O_2(n,s,x,b(\zeta )) - 1,\text { if }s = s(\sigma ), \\       
                            \tau^O_2(n,s,\ell ,0) - \tau^O_2(n,s,x,b(\zeta )),\text{ otherwise.}
      \end{cases}
\end{equation}

It follows from Equation (12) (and the proof of Equation (11)) that whenever $\left ( k^*_1, \ldots , k^*_m\right ) \in
\mathbf{K}_{b(\zeta ) - 1}$, there is some $\ell \in \overline{C}(n,\xi ,\zeta )$ such that (with $\sigma = \sigma_\ell$) for
all $1 \leq s \leq m$, the analogue of Equation (12) holds, with $k^*_s$ in place of $\text{I}k_{b(\zeta ) - 1, s}\left (\ell^*, 0\right )$.
Also for any such $\vec{k}^* = \left ( k^*_1, \ldots , k^*_m\right )$,  the frequency of $\vec{k}^*$ in
$\left ( \text{I}\vec{k}_{b(\zeta )-1}\left ( \ell^*\right )\vert 0 \leq \ell^* < 2^{(b(\zeta )-1)(M+1)}\right )$
is simply $\binom{b(\zeta )-1}{\vec{k}^*}$.

For $(n,\xi ,\zeta ,\vec{k}) \in \text{dom}\,\overline{c}, \sigma \in \{ 0, 1\}^D$, and $1 \leq s \leq m$, let: 

\begin{equation}
k^*_s(n,\xi ,\zeta ,\vec{k},\sigma ) = 
  \begin{cases}
          k_s - \tau^O_2(n,s,\xi ,b(\zeta )) - 1,\text{ if }s = s(\sigma ),\\
          k_s - \tau^O_2(n,s,\xi ,b(\zeta )),\text{ otherwise.}
  \end{cases}
\end{equation}
We say that $(n,\xi ,\zeta ,\vec{k},\sigma )$ is {\it bad }if $\left ( k^*_1,\ldots , k^*_m\right ) \not\in
\mathbf{K}_{b(\zeta )-1}$, where $k^*_s =
k^*_s(n,\xi ,\zeta ,\vec{k},\sigma )$, for $s = 1, \ldots , m$;
otherwise, $(n,\xi ,\zeta ,\vec{k},\sigma )$ is {\it good}.  Note that it is P-TIME decidable 
whether $(n,\xi ,\zeta ,\vec{k},\sigma )$ is bad, and that it is good iff for some $\ell \in
\overline{C}(n,\xi ,\zeta ),\ \vec{k} = \text{I}\vec{k}_n(\ell )$ and $\sigma = \sigma = \sigma_\ell$.
When $(n,\xi ,\zeta ,\vec{k},\sigma )$ is good, we also let 
$\vec{k}^*(n,\xi ,\zeta ,\vec{k},\sigma ) = \left ( k^*_1,\ldots , k^*_m\right )$, with each $k^*_s = 
k^*_s(n,\xi ,\zeta ,\vec{k},\sigma )$.

Then, we define:

\begin{equation}
c^*(n,\xi ,\zeta ,\vec{k},\sigma ) = 
     \begin{cases}
           0,\text{ if } (n,\xi ,\zeta ,\vec{k},\sigma )\text{ is bad,}\\
           \binom{b(\zeta )-1}{\vec{k}^*},\text{ where }\vec{k}^* = 
                \vec{k}^*(n,\xi ,\zeta ,\vec{k},\sigma ),\text{ otherwise.}
     \end{cases}
\end{equation}

It is then clear that $c^*$ is P-TIME and that:

\begin{equation}
\overline{c}(n,\xi ,\zeta ,\vec{k}) = 
\sum_{\sigma \in \{ 0, 1\}^D}c^*(n,\xi ,\zeta ,\vec{k},\sigma ).
\end{equation}

Thus, as required, $\overline{c}(n,\xi ,\zeta ,\vec{k})$ is also P-TIME.
\end{proof}

It is now, in the light of the proof of Lemma 2, that the ``raison d'$\hat{\text{e}}$tre'' for our choice of the $P_i$ finally appears
clearly.  For $\ell \in \overline{C}(n,\xi ,\zeta )$, for $i > b(\zeta )$, the $\left (\text{I}\vec{o}_n(\ell )\right )_i$ are
completely determined (by $\xi$), while for $i < b(\zeta )$, the $\left (\text{I}\vec{o}_n(\ell )\right )_i$ are
completely arbitrary.  It is only $\left (\text{I}\vec{o}_n(\ell )\right )_{b(\zeta)}$ which is ``partially determined'',
and the partial determination (by $\xi$) is completed by the $\sigma_\ell$.  This is the key to the proof of Equation (11),
from which the rest of the proof follows easily. 

Omitting the analogue of Remark 4, we proceed directly to the analogue of Corollary 3.
\begin{corollary}  
The relation $\text{\textnormal{RI}}B$ is \textnormal{P-TIME} decidable relative to $\tau^O_1$.
As a function of $(n,t,s)$ with $1 \leq s \leq \beta_{n,t}$, the function $b_{n,t,s}$ of 
\textnormal{Equation (7dii) (Definition 8)} is \textnormal{P-TIME} relative to 
$\tau^O_1$.
\end{corollary}
\begin{proof}   This follows from Lemma 2 and Lemma 2 of {[}\ref{DSS}{]}, because $b$ is enumerating
function of $\text{RI}B$ while $\beta$ is its cardinality function, so Lemma 2, above,
establishes that $\text{RI}B$ is tame.  Since it is tame, it is P-TIME decidable, by Lemma 2
of {[}\ref{DSS}{]}.
\end{proof}
\begin{theorem}  F is \textnormal{P-TIME} relative to $\tau^O_1$.
\end{theorem}
\begin{proof}  Let $\chi$ denote the characteristic function of the relation $\text{RI}B$.
Let $\Gamma$ be the three place relation on $\mathbb{N}$ such that $\Gamma (n,\ell , \ell ')$ holds iff
$n > 0,\ \ell, \ell ' < 2^{n(M+1)}$ and 
\begin{equation}
\beta (n,\text{\textnormal{IStep}}(n,\ell ),\ell ')\cdot
\chi(n,\text{\textnormal{IStep}}(n,\ell ),\ell ') =
\alpha (n,\text{\textnormal{IStep}}(n,\ell ),\ell ).
\end{equation}

Much as in {[}\ref{DSS}{]} (where the analogue of Equation (16) is Equation (7)), we show, first, that $\Gamma$ is
P-TIME decidable relative to $\tau^O_1$ and, given $n > 0$ and $\ell < 2^{n(M+1)}$ there is unique $\ell '$
such that $\Gamma (n,\ell , \ell ')$ holds.  We next argue 
that, as a function of $(n, \ell )$, the unique solution, $\ell '$, is
computable relative to $\tau^O_1$ in time polynomial in $n$.  Finally, we show that $\Gamma (n, \ell , \ell ')$
holds iff $\ell ' = F(n,\ell )$.  Of course this will suffice to show that .

By Proposition 1, the functions $\alpha$ and IStep are P-TIME relative to $\tau^O_1$, and by Lemma 2, the function
$\beta$ is, as well.  Since $\text{RI}B$ is P-TIME decidable, the function $\chi$ is P-TIME.  It is then clear that $\Gamma$
is P-TIME decidable.

Now suppose $n > 0$ and $\ell < 2^{n(M+1)}$.  Let $t = \text{IStep}(n,k)$ and, as in Remark 3, 
let $s = 1 + \ell - \text{SMC}^O(n, t)$.  Note that $\ell = a_{n,t,s}$ (and that
$s = \alpha (n, t, \ell )$, so, in particular, $\alpha (n, t, \ell ) > 0$).
By part (2) of Lemma 1, $s \leq \beta_{n,t}$, so $b_{n,t,s}$ is
defined and $b_{n,t,s} \in \text{I}B_{n,t}$.  Of course we have that
$s = \beta \left ( b_{n,t,s}\right )$.  It is then clear that
$\displaystyle{\Gamma \left ( n,\ell , b_{n,t,s}\right )}$ 
and so there is a solution $\ell '$ to $\Gamma (n,\ell , \ell' )$.  As was the case in {[}\ref{DSS}{]}, it is the fact of multiplying by
$\chi (n, t, \ell )$ that makes $\ell ' = b_{n,t,s}$ the unique solution, since this guarantees that if $\Gamma ( n, \ell , \ell'')$
holds then $\ell '' \in \text{I}B_{n,t}$.  Given this, and that $s = \alpha ( n,t,\ell ) = \beta (n,t,\ell '')$ it is then clear
that $\ell '' = \ell ' = b_{n,t,s}$.   This also makes it clear that, as a function of $(n, \ell )$, the unique solution, $\ell '$, is
computable relative to $\tau^O_1$ in time polynomial in $n$:  this is simply because the function $b$ is P-TIME
relative to $\tau^O_1$ (by Corollary 4), while $t$ and $s$ are also computable in time polynomial in $n$, relative to
$\tau^O_1$, because $t = \text{IStep}(n, \ell )$ and $s = 1 + \ell - \text{SMC}^O(n,t)$.
\end{proof} 
Combining Corollary 1 and Theorem 2 immediately gives:
\begin{theorem}  The trim strong triangular array representation of $\left\{ S^*_n\right \}$ corresponding to the 
sequence of admissible permutations encoded by $F$ is P-TIME.\qed
\end{theorem}
\subsection{Concluding Remarks: Retrospective Comparisons with \textnormal{{[}\ref{DSS}{]}}}
\label{subsec:3.3}
The list of results of {[}\ref{DSS}{]} for which we have presented no clear analogues in this paper is as follows:
\begin{enumerate}
\item  Propositions 1 and 2 in (3.2) of {[}\ref{DSS}{]},
\item  All of the numbered items of \S 4 of {[}\ref{DSS}{]},
\item  Corollary 4 and Remark 5 of (3.5) of {[}\ref{DSS}{]}.
\item  Proposition 3 of (3.3) of {[}\ref{DSS}{]}, Equation (8) of (3.4) of {[}\ref{DSS}{]}. 
\end{enumerate}

Of these, the results listed in items 1., 2., above, do have clear analogues.  Further, the proofs from {[}\ref{DSS}{]} go
over in a straightforward fashion.  Nevertheless, we have chosen to omit them on the grounds that the important point had already been
made in the setting of {[}\ref{DSS}{]}.  For the sake of completeness, we recall that Proposition 1 of {[}\ref{DSS}{]} shows that ``trimness does not
come for free'' by constructing a non-trim strong triangular array representations starting from a trim strong one, while Proposition
2 of {[}\ref{DSS}{]} is a ``non-persistence'' result, showing that in {\it any } sequence, $\left ( \pi_n\vert n \in \mathbb{Z}^+\right )$, of
admissible permutations and any $n$, it is never true that $\pi_{n+1}$ extends $\pi_n$ and that in {\it any } trim strong triangular
array representation of $\left \{ S^*_n\right\}$, it is never true that all of the random variables $R^*_{n,i}$ in the ``$n^{\text{th}}$ row reappear among the $R^*_{n+1,i}$ in the $n+1^{\text{th}}$ row''.  The material of \S 4 of {[}\ref{DSS}{]} dealt with the question of
constructing variants, $\left ( G_n\vert n \in \mathbb{Z}^+\right )$ and $\left ( H_n\vert n \in \mathbb{Z}^+\right )$ of
(that paper's version of) the sequence $\left ( F_n\vert n \in \mathbb{Z}^+\right )$.  The motivation for the variants is to
build some additional desirable properties into the individual $G_n$ and $H_n$; for example, the $G_n$ have the property
that if (in the notation of {[}\ref{DSS}{]}) $\ell \in A_{n,i} \cap B_{n,i}$, then $G_n(\ell ) = \ell$. 

The situation is different, unfortunately, regardiing the results listed in item 4. above.  These {[}\ref{DSS}{]} results involve obtaining a simple
expression for the base-2 exponential function in terms of {\it any } sequence  $\left ( \pi_n\vert n \in \mathbb{Z}^+\right )$, of
admissible permutations (Proposition 3 of {[}\ref{DSS}{]}) and an even simpler one in terms of the sequence 
$\left ( F_n\vert n \in \mathbb{Z}^+\right )$ (Equation (8) of {[}\ref{DSS}{]} for the {[}\ref{DSS}{]} version of this sequence).  There are no clear
analogues of these results, although there is a clear analogue of the method by which they were obtained in {[}\ref{DSS}{]}.  The basis of that method was the equivalence between being a power of 2 and having (Hamming) weight 1 (and so being a member of $B_{n,1}$).  The analogue of the approach taken in {[}\ref{DSS}{]} for a general sequence, $\left ( \pi_n\vert n \in \mathbb{Z}^+\right )$, of
admissible permutations is to consider the increment by 1 of the sum of all of the weight 1 elements, i.e, of all of the members of
$B_{n,1}$ (note:  $\gamma^O_{n,1}$ is still equal to $n$ as is easy to see).  The problem is that this may no longer
be the sum of the pure 2-powers: $\displaystyle{ \sum_{s=1}^n\pi_n(n)}$ is still the natural thing to consider, but it will depend on
$O$ in ways that make it difficult to identify what function this gives (though it seems quite plausible that it is at least
as complex as the exponential function with base 2).Here it is, once
again, the absence of a transparent connection between the notion of Weight and the patterns of 1's in binary expansions

Finally, concerning the results listed in item (3), above,
the content of Corollary 4 of {[}\ref{DSS}{]} is that the function SBC of that paper can be simply computed from $F$ and $\text{Inv}F$,
the (natural encoding of the) sequence of inverses, $\left\{ F^{-1}_n\right\}$.  The natural analogue, here would have
the function $\tau^O_1$ in place of SBC, and we conjecture that this is true, but we have yet to work this out.  Whether or not
SBC can be obtained from $F$ alone (or, somewhat more plausibly, from $\text{Inv}F$ alone is the subject of the speculative Remark 5 of {[}\ref{DSS}{]}; similar questions arise naturally here, but, here too, we have yet to work this out.

\begin{thebibliography}{}

\bibitem{Clote}\label{Clote}Clote, P., Kranakis, E.:  Boolean Functions and Computation Models.
Springer , Berlin-Heidelberg-New York (2002).

\bibitem{DG}\label{DG}\mbox{Dobri\ensuremath{\acute{\text{c}}}}, V. and Garmirian, P.  A New Direct Proof of the Central 
Limit Theorem.  arXiv:  1507.00357 (2015).

\bibitem{DSS}\label{DSS}\mbox{Dobri\ensuremath{\acute{\text{c}}}}, V., Skyers, Marina and Stanley, L.J.  Polynomial Computable Triangular Arrays for Almost Sure Convergence.  arXiv 1603.04896 (2016).

\bibitem{Papa}\label{Papa}Papadimitriou, C. H.:  Computational Complexity.
Addison-Wesley, Reading (1994).

\bibitem{Skorokhod}\label{Skorokhod}Skorokhod, A. V.: Limit
theorems for stochastic processes.Theory of Probability and its
Applications 1, 261-290 (1956).


\end{thebibliography}
\end{document}